\documentclass[11pt,a4paper]{article}
\usepackage[margin=1in]{geometry}
\usepackage[T1]{fontenc}
\usepackage[utf8]{inputenc}
\usepackage{xcolor}
\usepackage{subcaption}
\usepackage{bibunits}
\usepackage{tikz}
\usetikzlibrary{arrows.meta,calc,patterns}
\usepackage{hyperref}
\usepackage{amsmath,amssymb,amsfonts,amsthm,amscd, bm, mathtools}
\usepackage{mathrsfs}
\usepackage{lmodern}
\usepackage{titlesec}
 \titleformat{\subsection}[runin]
    {\normalfont\large\bfseries}{\thesubsection}{1em}{}
\newtheorem{theorem}{Theorem}[section]
\newtheorem{claim}{Claim}

\newtheorem{lemma}[theorem]{Lemma}
\newtheorem{corollary}[theorem]{Corollary}

\theoremstyle{definition}

\newtheorem{algorithm}[theorem]{Algorithm}
\newtheorem{remark}[theorem]{Remark}
\newtheorem{case}{Case}

\makeatletter
\newcommand\thankssymb[1]{\textsuperscript{\@fnsymbol{#1}}}
\makeatother
\hypersetup{
    colorlinks,
    linkcolor={red!60!black},
    citecolor={green!60!black},
    urlcolor={blue!60!black},
}

\title{Tree-cut decompositions for displaying undominated edge-ends}

\author{
Max Pitz
\and
Lucas Real\thanks{Supported by São Paulo Research Foundation
(FAPESP) through grant number 2025/00669-5.}
}

\begin{document}

\maketitle

\begin{abstract}
    We prove that every graph admits a linked, componental, rooted tree-cut
    decomposition of finite adhesion that displays all undominated edge-ends.

    As a first application, we deduce that this tree-cut decomposition also displays
    the edge-degrees of all undominated edge-ends. For locally finite graphs ---
    where every end is an undominated edge-end --- this yields a linked tree-cut
    decomposition of finite adhesion into \emph{finite} parts that displays all
    ends and their edge-degrees.

    As a second application, this latter tree-cut decomposition yields short,
    unified deductions of Thomassen's theorem on boundary-linked finite partitions,
    and of Bruhn and Stein's characterisation of Eulerian locally finite graphs
    in terms of even ends.
\end{abstract}

\section{Introduction}\label{sec:intro_new}

Tree-cut decompositions, introduced by Wollan \cite{WOLLAN2015}, are the edge-connectivity analogue of the classical tree-decompositions of Robertson and Seymour \cite{GraphMinorsIV}.
For many applications, however, one
wishes the adhesion sets of such a decomposition to faithfully reflect the
connectivity of the underlying graph; for tree-decompositions of finite
graphs, this is captured by Thomas's notion of a \emph{linked} tree-decomposition
\cite{THOMAS199067}. Two extensions of Thomas's theorem have since been pursued:
from finite to infinite graphs with tree-decompositions into finite parts by Albrechtsen, Jacobs, Knappe and Pitz
\cite{LinkedTreeDecompositions}, and from vertex separators to edge separators
(to \emph{linked tree-cut decompositions} for finite graphs) by Giannopoulou,
Kwon, Raymond and Thilikos \cite{GIANNOPOULOU20211}. The goal of this paper
is to combine these two threads, establishing linked tree-cut decompositions
for arbitrary infinite graphs.

\subsection{The main result}\label{subsec:main_new}
All graphs in this paper are multigraphs with parallel edges allowed.
To state our main theorem precisely, a (rooted) \emph{tree-cut decomposition}
$(T,\mathcal{V})$ of a graph $G$ consists of a rooted tree $T$ together with
a family $\mathcal{V} = (V_t)_{t \in T}$ of pairwise disjoint bags
$V_t \subseteq V(G)$ with $\bigsqcup_{t \in T} V_t = V(G)$; empty bags are
allowed. Each edge $e \in E(T)$ separates $T$ into two components $T_1, T_2$
and yields the corresponding \emph{adhesion set}
$F_e := E_G\bigl(\bigcup_{t \in T_1} V_t,\, \bigcup_{t \in T_2} V_t\bigr)
\subseteq E(G)$; the decomposition has \emph{finite adhesion} if every $F_e$
is finite. Writing $\leq$ for the tree-order on $T$ rooted at
$r := \mathrm{root}(T)$, and $C_t := \bigcup_{s \geq t} V_s$ for $t \in T$,
every finite-adhesion tree-cut decomposition induces a natural map
$\varphi \colon \Omega_E(G) \to V(T) \cup \Omega(T)$ (defined precisely in
Section~\ref{sec:prelims}). We say that $(T,\mathcal{V})$:

\begin{itemize}
    \item is \emph{componental} if $G[C_t]$ is non-empty and connected for
        every $t \in T$;
    \item is \emph{linked} if for every finite path $t_0 t_1 \dots t_n$ in $T$
        with $t_0 < t_1 < \dots < t_n$ and $n \geq 1$, there are
        $\min\{|F_{t_{i-1}t_i}| : 1 \leq i \leq n\}$ edge-disjoint paths in
        $G$ from $F_{t_0 t_1}$ to $F_{t_{n-1}t_n}$;
    \item \emph{displays} a set $\Psi \subseteq \Omega_E(G)$ of edge-ends if
        $\varphi^{-1}(\Omega(T)) = \Psi$ and the restriction
        $\varphi|_\Psi \colon \Psi \to \Omega(T)$ is a bijection.
\end{itemize}

\begin{theorem}\label{Main}
    Every graph admits a linked, componental, rooted tree-cut decomposition
    of finite adhesion that displays all undominated edge-ends.
\end{theorem}

Dropping the linkedness requirement, Theorem~\ref{Main} reduces to a
recent result of Kurkofka \cite{Kurkofka2022} (recovered with a shorter proof
and a sharper conclusion by the second author \cite{TreeCutPartition}). It is
the linked property in Theorem~\ref{Main} that powers the applications in
Subsections \ref{subsec:edgedegrees_new}--\ref{subsec:eulerian_new} below; this property
cannot, however, be relaxed to its natural unrooted strengthening. Mirroring
the counterexamples of Albrechtsen, Jacobs, Knappe and Pitz in
\cite{CounterexamplesLinkedTreeDecompositions} for the vertex setting, we
exhibit in Appendix~\ref{section:appendix} a locally finite graph for which
no tree-cut decomposition displaying all its ends is linked in the unrooted sense.
For a related result about displaying the undominated vertex-ends by linked tree-decompositions, see \cite{albrechtsen2025displaying}.

We hope to convince the reader of the usefulness of Theorem~\ref{Main}
by demonstrating in Subsections \ref{subsec:edgedegrees_new}--\ref{subsec:eulerian_new}
below the surprisingly powerful interplay of being linked, componental, and
of finite adhesion through three applications.

\subsection{Displaying edge-degrees}\label{subsec:edgedegrees_new}

As our first application of Theorem~\ref{Main}, the witnessing tree-cut decomposition displays the edge-degrees of all undominated edge-ends, resolving the edge-analogue of a question of Halin from 1977 \cite{halin1977systeme} asking, for a locally finite graph, for a tree-decomposition that simultaneously displays all ends and their combined degrees. Recall that the \emph{edge-degree} of an edge-end $\omega \in \Omega_E(G)$ is defined as
\[
\mathrm{deg}(\omega) = \sup\bigl\{|\mathcal{R}| : \mathcal{R} \subseteq \omega
\text{ is a family of edge-disjoint rays}\bigr\} \in \mathbb{N} \cup \{\infty\}.
\]

\begin{theorem}\label{Edgedegree}
    Every graph $G$ admits a tree-cut decomposition $(T,\mathcal{V})$ of finite
    adhesion that displays all undominated edge-ends together with their
    edge-degrees: for each displayed $\omega$ with corresponding ray
    $R_\omega = t_0 t_1 t_2 \dots$ in $T$ starting at the root,
    $\mathrm{deg}(\omega) = \liminf_n |\delta(C_{t_n})|$.
\end{theorem}

For locally finite graphs, every end is an undominated edge-end, and we deduce
the following strengthened version:

\begin{theorem}\label{EdgedegreeLocFin}
    Every locally finite connected graph $G$ admits a linked tree-cut
    decomposition of finite adhesion into finite parts, with the
    decomposition tree itself locally finite, that displays all ends of $G$
    and their edge-degrees.
\end{theorem}

The reduction from Theorem~\ref{Main} to Theorem~\ref{Edgedegree}
is short: applying the linkedness property between minimal adhesion sets on the ray $R_\omega$ forces the edge-degree to match the $\liminf$ along $R_\omega$. The additional finite-bags conclusion of
Theorem~\ref{EdgedegreeLocFin} then follows from K\H{o}nig's result that every infinite locally finite graph contains a ray. Both reductions are carried out in Section~\ref{sec:edgedegrees}.

\subsection{Boundary-linked finite partitions: Thomassen's theorem}%
\label{subsec:thomassen_new}

A \emph{region} of $G$ is a set
$C \subseteq V(G)$ such that $G[C]$ is connected and the edge-boundary
$\delta(C) \subseteq E(G)$ is finite; we write $\overline{C}$ for the subgraph
induced by $C$ together with the edges of $\delta(C)$. A region $C$ is
\emph{boundary-linked} to an edge-end $\omega \in \Omega_E(G)$ if $\overline{C}$
contains a family $\mathcal{R} \subseteq \omega$ of $|\delta(C)|$ edge-disjoint
rays starting at the edges of $\delta(C)$. 
Thomassen proved the following finite-partition theorem for locally finite graphs \cite{Thomassen}, which we now re-derive as a corollary of Theorem~\ref{Main}:

\begin{corollary}[\cite{Thomassen}, Theorem 1]\label{ThomassenCor}
    Let $G$ be a connected and locally-finite graph. If $A_0\subseteq V(G)$
    has finite boundary, then $V(G)\setminus A_0$ admits a finite partition
    $\mathcal{C}$ whose elements are either singletons or boundary-linked
    regions.
\end{corollary}

\begin{proof}
    The result is trivial for finite $G$, so assume $G$ is infinite. Since
    $A_0$ has finite edge-boundary, contracting $A_0$ to a single vertex
    preserves local finiteness, and we may assume without loss of generality
    that $A_0 = \{v_0\}$ for some $v_0\in V(G)$.

    Let $(T,\mathcal{V})$ be a linked and componental tree-cut decomposition
    of finite adhesion for $G$ provided by Theorem~\ref{Main}. According
    to Lemma~\ref{ConclusionsLocallyFinite}, $T$ is locally finite and each
    bag $V_t$ is finite.

    Then, let $t_0 \in T$ denote the node satisfying $v_0\in V_{t_0}$.
    Writing $P:=\lceil t_0\rceil$, Lemma~\ref{LinkedDisplay} ensures that
    each ray $R$ in $T$ starting at a node from $P$ contains a node
    $t_R\in T\setminus P$ such that $C_{t_R}$ is boundary-linked to the
    edge-end $\varphi_{\mathcal{T}}^{-1}([R]_E)$. Choose one such $t_R$ for
    each $R$, and let $T'$ be the connected component of
    \[T\setminus\{t_R\colon R\text{ a ray of }T\text{ starting at some node of }P\}\]
    containing $P$. Then $T'$ is locally finite and rayless, hence finite
    by K\H{o}nig's Lemma.

    Let $S := \{t\in T\setminus T'\colon t\text{ has a neighbour in }T'\}$.
    Since $T'$ is a rooted subtree of $T$ and the only $T$-neighbours of $T'$
    outside $T'$ are the deleted nodes, $S$ is a finite antichain in the
    tree order of $T$. In particular, each $t\in S$ equals $t_R$ for some
    ray $R$, and
    \[
        T\setminus T' \;=\; \bigsqcup_{t\in S}\lfloor t\rfloor.
    \]
    Moreover, $S':=\bigsqcup_{t\in T'}V_t$ is finite by
    Lemma~\ref{ConclusionsLocallyFinite} and contains $v_0$, since
    $t_0\in P \subseteq T'$. Therefore,
    \[
        \{C_t\colon t\in S\}\;\cup\;\{\{v\}\colon v\in S'\setminus\{v_0\}\}
    \]
    is a finite partition of $V(G)\setminus\{v_0\}$ into boundary-linked
    regions and singletons.
\end{proof}

\subsection{Eulerian locally finite graphs: Bruhn--Stein's characterisation}%
\label{subsec:eulerian_new}

Following Bruhn and Stein \cite{BruhnStein}, an edge-end
$\omega \in \Omega_E(G)$ of a locally finite graph $G$ is \emph{even} if
there is some finite set of vertices $S^*\subseteq V(G)$ such that, for every
finite $S\subseteq V(G)$ with $S\supseteq S^*$, the maximum size of an
edge-disjoint family $\mathcal{R}\subseteq \omega$ starting at vertices of
$S$ is even; otherwise $\omega$ is an \emph{odd} end. Following the
topological approach of Diestel and Kühn in \cite{Diestel2004}, an Eulerian
locally finite graph is one whose Freudenthal compactification admits a
continuous image of the unit circle traversing each edge exactly once;
equivalently, every finite cut has even size, recovering the familiar
finite-graph criterion. 
Combined with the classical even-degree condition, Bruhn and Stein obtained the following characterisation, which we re-derive as a corollary of Theorem~\ref{Main}:

\begin{corollary}[\cite{BruhnStein}, Theorem 4]\label{EulerianCor}
    A locally finite, connected graph $G$ is Eulerian if, and only if, all
    vertices have even degree and every (edge-)end
    $\omega\in \Omega(G)= \Omega_E(G)$ is even.
\end{corollary}

\begin{proof}[Proof of Corollary~\ref{EulerianCor} from Theorem~\ref{Main}]
We prove only the hard implication. Assume all vertices of a locally finite,
connected graph $G$ have even degree but that some $X\subseteq V(G)$ has cut
$\delta(X)$ of finite odd size.
Let $\tilde{G}$ be the graph obtained from $G$ by contracting $X$ to a single
vertex $v_X \in V(\tilde{G})$ of degree $|\delta(X)|$.
Since $|\delta(X)|$ is finite and $G$ is locally-finite, so is $\tilde{G}$,
and $v_X$ is its unique vertex of odd degree.
    Let $(T,\mathcal{V})$ be a linked and componental tree-cut decomposition
    of finite adhesion for $\tilde{G}$ provided by Theorem~\ref{Main},
    and consider $t_*\in T$ as the node so that $v_X\in V_{t_*}$.
    By Lemma~\ref{ConclusionsLocallyFinite}, $T$ is locally-finite and each
    bag $V_t$ is finite.

    If $K\subseteq T$ denotes the set of $\leq${-}minimal nodes of
    $T\setminus \lceil t_*\rceil$, the family
    $\mathcal{C}_*:=\{C_s \colon s \in K\}$ comprises pairwise disjoint
    regions of $G$.
    Hence, the graph $G_*$ obtained from $\tilde{G}$ after contracting each
    subgraph in $\mathcal{C}_*$ to a single vertex is again finite and whose
    vertices from $\bigcup_{t\leq t_*}V_t\setminus \{v_X\}$ have even degree
    by hypothesis.
    Due to the Handshaking Lemma (Proposition 1.2.1 in \cite{Diestel2004}),
    the odd degree of $v_X$ implies the existence of $t_0\in K$ such that
    $|\delta(C_{t_0})|$, which is the degree of the contraction vertex
    corresponding to $C_{t_0}$, is also odd.

We now construct the rest of the ray inductively: we shall produce
$t_0<t_1<t_2<\cdots$ in $T$ with $|\delta(C_{t_i})|$ odd for every $i\geq 1$.
Suppose $t_0,\ldots,t_n$ have been defined.
Let $\tilde{G}_n$ be the finite graph obtained from $\tilde{G}$ by contracting
each member of the disjoint family
$$\mathcal{C}_n:=\{C_s\colon s \in \mathrm{succ}(t_n)\}\cup \{D_n\},\qquad D_n:=V(\tilde{G})\setminus C_{t_n},$$
to a single vertex, so that
$V(\tilde{G}_n)=V_{t_n}\cup\{v_C\colon C\in\mathcal{C}_n\}$.
By the Handshaking Lemma once more, $\tilde{G}_n$ has an even number of
vertices of odd degree.
Every vertex of $V_{t_n}$ inherits its degree from $G$ and hence has even
degree. Among the remaining vertices of $\tilde{G}_n$, $v_{D_n}$ has 
degree $|\delta(D_n)|=|\delta(C_{t_n})|$, which is odd by the previous
paragraph's conclusion if $n=0$ and by the inductive hypothesis if $n\geq 1$.
Hence, Handshaking forces some $t_{n+1}\in\mathrm{succ}(t_n)$ with
$v_{C_{t_{n+1}}}$ of odd degree in $\tilde{G}_n$.
The edges of $\tilde{G}_n$ incident with $v_{C_{t_{n+1}}}$ are precisely those
of $\delta(C_{t_{n+1}})$ in $\tilde{G}$, so $|\delta(C_{t_{n+1}})|$ is odd,
completing the induction.

Let $R:=t_0t_1t_2\ldots$ be the ray produced by this recursion, and let
$\omega:=\varphi_{\mathcal{T}}^{-1}([R]_E)\in\Omega_E(\tilde{G})$ be the
(unique, undominated) edge-end of $\tilde{G}$ living in the end $[R]_E$ of $T$.
For every $n\geq 0$ we have $v_X\notin C_{t_n}$, so that
$C_{t_n}\subseteq V(G)\setminus X$ and every ray of $\tilde{G}$ eventually
in $C_{t_n}$ is also a ray of $G$; in this way $\omega$ determines an edge-end
of $G$, which we continue to denote by $\omega\in\Omega_E(G)$.

We claim that $\omega$ is odd. Let $S^*\subseteq V(G)$ be any finite candidate
witness. By Lemma~\ref{LinkedDisplay}, we may fix $n\geq 1$ so large that
$C_{t_n}$ is disjoint from $S^*$ and boundary-linked to $\omega$. Set
$S:=S^*\cup N_G(C_{t_n})$, the finite subset of $V(G)$ obtained from $S^*$ by
adjoining the $G$-neighbours of $C_{t_n}$ lying outside $C_{t_n}$. Any
edge-disjoint family $\mathcal{R}\subseteq\omega$ of rays of $G$ starting at
vertices of $S$ has each ray using an edge of $\delta(C_{t_n})$, so
$|\mathcal{R}|\leq|\delta(C_{t_n})|$; conversely, boundary-linkedness of
$C_{t_n}$ to $\omega$ provides such a family of exactly this size. Hence the
maximum size of an edge-disjoint family $\mathcal{R}\subseteq \omega$ starting
at the vertices of $S$ is precisely the odd number $|\delta(C_{t_n})|$,
showing that $\omega$ is odd.
\end{proof}

\subsection{How this paper is organised}\label{subsec:organisation_new}

In Section~\ref{sec:prelims} we collect preliminaries on edge-connectivity,
edge-ends, and tree-cut decompositions, and recall the unlinked predecessor
of Theorem~\ref{Main} due to Kurkofka \cite{Kurkofka2022} and the second
author \cite{TreeCutPartition}. In Section~\ref{sec:TopVSComb} we derive
Theorems~\ref{Edgedegree} and~\ref{EdgedegreeLocFin} from
Theorem~\ref{Main}. The construction of the linked tree-cut decomposition
of Theorem~\ref{Main} occupies
Sections~\ref{sec:boundarylinked} and \ref{sec:proof}: 
Section~\ref{sec:boundarylinked}
establishes a submodularity / uncrossing lemma for boundary-linked regions,
Section~\ref{subsec:algorithm} develops a cutting-family algorithm,
Section~\ref{sec:successor} partitions the undominated edge-end structure, and
Section~\ref{subsec:buildingup} assembles the tree-cut decomposition.
Appendix~\ref{section:appendix} presents the counterexample to the
strongly linked strengthening.

\section{Preliminaries}
\label{sec:prelims}

\subsection{Edge-connectivity and edge-ends} 
The edge-ends of a fixed infinite graph $G$ were first discussed by Hahn, Laviolette and \v{S}ira\v{n} in \cite{EdgeEnds} after naturally modifying the well-known definition of (vertex-)ends from Halin in \cite{Halin1964}. 
In both cases, these objects comprise equivalence classes of one-way infinite paths of $G$ after identifying those that are either infinitely vertex-connected or infinitely edge-connected, depending on the chosen structure. 
As well as its formalization, basic properties regarding the classical (vertex-)end space $\Omega(G)$ of $G$ can be found in the textbook \cite[Chapter 8]{DiestelBook}.

In order to first compile basic terminology on edge-connectivity, however, recall the notations $E(A,B) := \{xy\in E(G) : x\in A, y\in B\}$, $E(C):=E(C,C)$ and $\delta(C):=E(C,V(G)\setminus C)$ for given sets of vertices $A,B,C\subseteq V(G)$. 
In particular, $\delta(C)$ is often referred to as the \textbf{cut} or \textbf{boundary} of $C$, while the subgraph induced by $C$ and this set of edges may be written as $\overline{C}$. 
Then, $N(C):=\overline{C}\setminus C$ shall be used to denote the set of endpoints in $V(G)\setminus C$ of the edges from $\delta(C)$. 
Furthermore, $C \subseteq V(G)$ is called a \textbf{region} of $G$ if $G[C]$ is connected and its boundary $\delta(C)$ is finite. On the other hand, two subgraphs $H$ and $H'$ of $G$ are said to be \textbf{separated} by a finite set of edges $F\subseteq E(G)$ if there is no path in $G\setminus F$ containing both a vertex of $H$ and one of $H'$.

Now, aiming to properly revisit the literature on edge-ends, recall that a \textbf{ray} in $G$ is any one-way infinite path. 
When presenting it in terms of its vertex set as $R:=v_0v_1v_2\dots$, we also refer to $v_0$ and $v_0v_1$ as its \textbf{starting vertex} and \textbf{starting edge} respectively. 
In addition, we call the infinite connected subgraphs of $R$ its \textbf{tails}, which then have the form $v_nv_{n+1}v_{n+2}\dots$ for some $n\in\mathbb{N}$. 
Considering that, a second ray $S$ of $G$ is said to be \textbf{edge-equivalent} to $R$ if no tail of $S$ can be separated from a tail of $R$ by a finite set of edges. 
This now defines an equivalence relation over the family of all rays in $G$, whose quotient shall then be written as $\Omega_E(G):=\{[R]_E: R\text{ is a ray of }G\}$. 
Following the original concept from \cite{EdgeEnds}, the \textbf{edge-end} of the ray $R$ corresponds to its edge-equivalence class $[R]_E$. 
In particular, for every finite set of edges $F\subseteq E(G)$, the connected component of $G\setminus F$ containing a tail of $R$ is uniquely determined upon this edge-end and, hence, can be denoted simply by $C_E(F,[R]_E)$. If $v\notin C_E(F,[R]_E)$, we say that such a vertex $v\in V(G)$ is \textbf{separated} from $[R]_E$ by $F$.

In general, we may use lower-case Greek letters such as ``$\omega$'' and ``$\varepsilon$'' for referring to the elements of $\Omega_E(G)$, besides also saying that some $\omega \in \Omega_E(G)$ \textbf{lives in} a subgraph $C$ of $G$ if all its representatives have tails in $C$. 
When $C$ is actually a region, this is equivalently written as $C_E(\delta(C),\omega) = C$. 
Moreover, following a definition introduced by Thomassen in \cite{Thomassen} (although there for locally-finite graphs), we say that $C$ is \textbf{boundary-linked} to $\omega$ if there is a family $\mathcal{R}\subseteq \omega$ of $|\delta(C)|$ many edge-disjoint paths in $\overline{C}$ starting at the edges from $\delta(C)$. 
In particular, each such edge must belong to precisely one ray of $\mathcal{R}$.

A vertex $v\in V(G)$ \textbf{edge-dominates} a ray $R$ of a graph $G$ if it cannot be separated from any tail of $R$ by a finite set of edges. 
Naturally, $\omega:=[R]_E$ is hence called an edge-dominated (or simply \textbf{dominated}) edge-end of $G$, while $v$ is referred to as one of its \textbf{edge-dominating vertices}.

\subsection{Tree-cut decompositions}\label{subsec:tree-cut} 
We shall now recall a tree-like partition introduced by Wollan in \cite{WOLLAN2015} while studying  graphs not containing a fixed immersion. 
Following his definition, a \textbf{tree-cut decomposition} for a graph $G$ is a pair $(T,\mathcal{V})$ where a rooted tree $T$ fits as an index set to a \textit{near} partition $\mathcal{V}:=\{V_t: t\in T\}$ of $V(G)$, in the sense that $\bigsqcup_{t\in T}V_t = V(G)$ is a disjoint union but we may have $V_t = \emptyset$ for some nodes $t\in T$. 
Furthermore, we also write $\leq$ for the natural tree-order of $T$, while the elements of $\mathcal{V}$ are called its \textbf{bags}. 
For each $t\in T$, we particularly fix the notations $\lfloor t\rfloor:=\{s\in T: s \geq t\}$ and $C_t:=\bigcup_{s\geq t}V_s$, from where $C_s\subseteq C_t$ holds for every $s\in \lfloor t\rfloor$. 
Incidentally, we may refer to the set of \textbf{successors} of $t$ as $\mathrm{succ}(t):=\{s\in \lfloor t\rfloor : st\in E(T)\}$, then remarking that $\mathrm{succ}(t)$ comprises pairwise incomparable elements regarding $\leq$ and, hence, $C_s\cap C_{s'}=\emptyset$ for two distinct $s,s'\in \mathrm{succ}(t)$.

On the other hand, the \textbf{adhesion set} of $(T,\mathcal{V})$ corresponding to an edge $e\in E(T)$ is defined as $F_e:=E\left(\bigcup_{s\in T_1}V_s, \bigcup_{s'\in T_2}V_{s'}\right)$, where $T_1$ and $T_2$ denote the two connected components of $T\setminus e$.
As a notion particularly concerning infinite graphs, the pair $(T,\mathcal{V})$ itself is said to have \textbf{finite adhesion} if $F_e$ is finite for every $e\in E(T)$. 
Comparing with Chapter 12 from the textbook \cite{DiestelBook}, this terminology reads similarly to the classical  \textit{tree-decompositions} that Robertson and Seymour formalized in \cite{GraphMinorsIV} to their theory of forbidden minors. 
Incidentally, the tree-\textit{cut} decomposition claimed in Theorem \ref{Main} shall indeed mirror the role played by the tree-decompositions from the inspiring main result due to Albrechtsen, Jacobs, Knappe and Pitz in \cite{LinkedTreeDecompositions}. 
In our context, however, their linked, tight and componental properties are now slightly modified with regard to $(T,\mathcal{V})$ as above, which is again said to be:

\begin{itemize}
    \item \textbf{componental} if $C_t=\bigcup_{s \geq t}V_s$ induces a non-empty connected subgraph of $G$ for every $t\in T$;
    \item \textbf{linked} if, for every finite path $t_0t_1t_2\dots t_n$ in $T$ with $t_0<t_1<t_2<\dots <t_n$ and $n\geq 1$, there are $\kappa:=\min\{|F_{t_{i-1}t_i}|: 1 \leq i \leq n\}$ many edge-disjoint paths in $G$ connecting the edges from $F_{t_0t_1}$ to those of $F_{t_{n-1}t_n}$. 
\end{itemize}

Note that for an edge $e=st \in E(T)$ with $s<t$, the corresponding adhesion set $F_e$ is precisely the cut $\delta(C_t)$. 
Then, the cardinal $\kappa$ from the above second bullet point is alternatively written as $\kappa = \min\{|\delta(C_{t_i})|: 1 \leq i \leq n\}$. 
This encodes a Menger-type duality: under this condition, $\kappa$ describes both the minimum size for a set of edges separating $\delta(C_{t_1})$ from $\delta(C_{t_n})$ and the maximum size of an edge-disjoint system of paths connecting them. 

Comparing with the literature on tree-decompositions, this is then the edge-version of the linkedness criteria studied by Thomas in \cite{THOMAS199067}.

\medskip

We now consider how the edge-ends of a graph $G$ interact with a  tree-cut decomposition $\mathcal{T}=(T,\mathcal{V})$ of finite adhesion. 
As every edge $e=t_1t_2 \in E(T)$ induces a finite cut $X_e := E_G(\bigcup_{t \in T_1} V_t , \bigcup_{t \in T_2} V_t)$ in $G$, any edge-end of $G$ has to choose one component $T_1$ or $T_2$ of $T-e$, and we may visualise this decision by orienting $e$ accordingly. Then for a fixed end, all the edges point either towards a unique node or towards a unique end of $T$.
In this way, each edge-end of $G$ \textbf{lives} in a part of $\mathcal{V}$ or in an end of $T$, and we may encode this correspondence by a map $\varphi_\mathcal{T}  \colon \Omega_E(G) \to V(T) \cup \Omega(T)$. 
A tree-cut decomposition $\mathcal{T} = (T,\mathcal{V})$ \textbf{displays} a set $\Psi$ of edge-ends of $G$ if in every end of $T$ there lives just one edge-end of $G$ and that edge-end belongs to $\Psi$, and conversely, every edge-end in $\Psi$ lives in some end of $T$.
In particular, for an edge-end $\omega \in \Psi$, then there will exist a unique rooted ray $R_{\omega}:=t_0t_1t_2\dots$ starting at the root $t_0$ of $T$ such that $\omega$ lives in $C_{t_n}=\bigcup_{s\geq t_n}V_s$ for every $n\in \mathbb{N}$.

As announced in the introduction, if we do not insist on the linkedness requirement, then our main Theorem~\ref{Main} reduces to a recent result of Kurkofka \cite{Kurkofka2022} (recovered with a shorter proof
by the second author \cite{TreeCutPartition}). 

\begin{theorem}
\label{DisplayKurkofka}
    Every graph $G$ admits a tree-cut decomposition of finite adhesion $(T,\mathcal{V})$ displaying the undominated edge-ends of $G$ and such that $G\left[\bigcup_{t\in T'}V_t\right]$ is a region of $G$ for every region $T'$ of $T$.
\end{theorem}

Although stated in \cite{TreeCutPartition} for simple graphs, Theorem \ref{DisplayKurkofka} extends to multigraphs by a routine subdivision argument (subdivide each edge by a new vertex and pull back the resulting tree-cut decomposition). We may therefore freely apply it to multigraphs throughout.

\section{Displaying undominated edge-ends}\label{sec:TopVSComb}
\label{sec:edgedegrees}

In this section, we prove that Theorem \ref{Main} implies Theorems~\ref{Edgedegree} and \ref{EdgedegreeLocFin}.
A first immediate consequence of the definition of linkedness is the following lemma, showing how the edge-disjoint path systems provided by linkedness yield families of edge-disjoint rays converging into a displayed edge-end:

\begin{lemma}\label{LinkedDisplay}
    Let $\mathcal{T}=(T,\mathcal{V})$ be a linked and componental tree-cut decomposition of finite adhesion for a graph $G$ displaying a set of edge-ends $\Psi$.
    Suppose that $R:=t_0t_1t_2\dots$ is a ray in $T$ starting at its root, and that $n_0\in \mathbb{N}$ satisfies $|\delta(C_{t_m})|\geq |\delta(C_{t_{n_0}})|$ for every $m\geq n_0$.
    Then $C_{t_{n_0}}$ is boundary-linked to the edge-end $\omega :=\varphi_{\mathcal{T}}^{-1}([R]_E)$.
\end{lemma}
\begin{proof}
    Choose inductively a strictly increasing sequence $n_0 < n_1 < n_2 < \cdots$ (starting with the given $n_0$) such that
    \[
        |\delta(C_{t_{n_{k+1}}})| \;=\; \min_{j>n_k}|\delta(C_{t_j})|;
    \]
    this minimum exists, since the set $\{|\delta(C_{t_j})|\colon j>n_k\}$ is a non-empty subset of $\mathbb{N}$. 
    Let us write $C_k:=C_{t_{n_k}}$ and $m_k:=|\delta(C_k)|$. The hypothesis on $n_0$ together with the above choice yields $m_0 \le m_1 \le m_2 \le \cdots$.

    By the linkedness assumption applied to the path $t_{n_k}t_{n_k+1}\cdots t_{n_{k+1}}$ in $T$, there is a family $\mathcal{P}_k$ of $\min\{|\delta(C_{t_i})|\colon n_k<i\le n_{k+1}\}$ edge-disjoint paths in $G$ connecting $\delta(C_k)$ to $\delta(C_{k+1})$. 
    Then each edge $e \in \delta(C_k)$ belongs to a unique path $P_e \in \mathcal{P}_k$, and $E(P_e)$ intersects $\delta(C_{k+1})$ precisely in its last edge, since $C_{k+1}\subseteq C_k$.
    So by concatenating suitable paths from $(\mathcal{P}_k)_{k \geq 1}$, this argument shows how $\bigcup_{k\geq 1}\bigcup \mathcal{P}_k$ contains a family $\mathcal{R}$ of $|\delta(C_0)|$ many edge-disjoint one-way paths in $\overline{C_0}$ starting at the edges from $\delta(C_0)$. In fact, for every $k\in\mathbb{N}$, the region $C_k$ contains a tail of each $R\in\mathcal{R}$ arising from this construction, which is then indeed an infinite subgraph of $G$. In other words, $\mathcal{R}$ is an edge-disjoint family of rays in $\overline{C_{t_{n_0}}}$ such that $\mathcal{R}\subseteq \omega:=\varphi_{\mathcal{T}}^{-1}([R]_E)$ by definition of $\varphi_{\mathcal{T}}$.
  \end{proof}

The following result clearly implies Theorem~\ref{Edgedegree}:

\begin{theorem}\label{DisplayEdgeDegrees}
Let $(T,\mathcal{V})$ be a componental and linked tree-cut decomposition of finite adhesion for a graph $G$ displaying a set of edge-ends $\Psi$.
The edge-degree of a given $\omega \in \Psi$ is given by $$\mathrm{deg}(\omega)=\liminf_{n\in\mathbb{N}}|\delta(C_{t_n})|,$$ where $R_\omega:=t_0t_1t_2\dots$ is the unique ray in $T$ starting at the root $t_0\in T$ that corresponds to $\omega$.
\end{theorem}
\begin{proof}
    To reduce notation, we write $C_n :=C_{t_n}$ for each $n\in\mathbb{N}$; this is indeed a region in $G$ by the componental property of $(T,\mathcal{V})$. Set $\kappa:=\liminf_{n\in\mathbb{N}}|\delta(C_n)|$.

    We first argue that $\mathrm{deg}(\omega)\leq \kappa$. If $\kappa = \infty$ this is trivial, so assume $\kappa<\infty$, and suppose for a contradiction that $\mathrm{deg}(\omega)\geq \kappa + 1$. Fix a finite edge-disjoint family $\mathcal{R}\subseteq \omega$ of size $\kappa+1$, and choose $n_0\in\mathbb{N}$ large enough that $C_{n_0}$ contains no starting vertex of any ray from $\mathcal{R}$ and that $|\delta(C_j)|\geq \kappa$ for every $j> n_0$.
     Pick some $n>n_0$ with $|\delta(C_n)| = \kappa$. Each ray in $\mathcal{R}$ starts outside $C_n$ and has a tail in $C_n$ (since $\omega$ lives in $C_n$ by definition of $\varphi_{\mathcal{T}}$), so it uses at least one edge of $\delta(C_n)$; edge-disjointness of $\mathcal{R}$ then gives an injection from $\mathcal{R}$ into $\delta(C_n)$, yielding $\kappa+1 = |\mathcal{R}|\leq|\delta(C_n)|=\kappa$, a contradiction.

    For the reverse inequality, fix any $n_0\in\mathbb{N}$ and choose $n > n_0$ with $|\delta(C_n)| = \min_{j>n_0}|\delta(C_j)|$ (which exists as $|\delta(C_\cdot)|$ is $\mathbb{N}$-valued). Then $|\delta(C_n)|\leq |\delta(C_j)|$ for every $j\geq n$, so Lemma \ref{LinkedDisplay} applies and yields that $C_n$ is boundary-linked to $\omega$. In particular, there is an edge-disjoint family in $\omega$ of $|\delta(C_n)|$ many rays, whence $\mathrm{deg}(\omega)\geq |\delta(C_n)| = \min_{j>n_0}|\delta(C_j)|$. Taking the supremum over $n_0$ yields $\mathrm{deg}(\omega)\geq \liminf_{n\in\mathbb{N}} |\delta(C_n)| = \kappa$, as claimed.
\end{proof}

And our next lemma  clearly implies Theorem~\ref{EdgedegreeLocFin}:

\begin{lemma}\label{ConclusionsLocallyFinite}
    Let $(T,\mathcal{V})$ be a tree-cut decomposition of finite adhesion for a locally-finite connected graph $G$. If $(T,\mathcal{V})$ is componental and displays the (edge-)ends of $G$, then $T$ is locally-finite and each bag in $\mathcal{V}$ is finite. 
\end{lemma}
\begin{proof}
    Fix $t\in T$. Let $\tilde{C}_t$ be the graph on vertex set $V_t \cup \{v_s\colon s\in \mathrm{succ}(t)\}$ obtained from $G[C_t]$ by contracting each region $C_s$ ($s\in \mathrm{succ}(t)$) to a single vertex $v_s$. Under this identification, we then have $E(\tilde{C}_t)\subseteq E(G)$. Each $v_s$ has degree $\deg_{\tilde{C}_t}(v_s) \leq |\delta(C_s)| < \infty$ by finite adhesion, and each $v\in V_t$ satisfies $\deg_{\tilde{C}_t}(v)\leq \deg_G(v)<\infty$ by local finiteness of $G$; hence $\tilde{C}_t$ is locally finite. Moreover, $\tilde{C}_t$ is connected, being a contraction of the graph $G[C_t]$ (which is connected by the componental property of $(T,\mathcal{V})$).

    We claim that $\tilde{C}_t$ is finite. Suppose otherwise and, by K\H{o}nig's Lemma, $\tilde{C}_t$ contains a ray $\tilde{R} := v_0 v_1 v_2 \dots$ Its edges $e_n := v_nv_{n+1}$ ($n\geq 0$) all lie in $E(G)$; we pull $\tilde{R}$ back to a ray in $G$ as follows. Write $e_n = x_n y_n\in E(G)$, with the convention that $x_n$ represents $v_n$ and $y_n$ represents $v_{n+1}$ (so $x_n = v_n$ if $v_n\in V_t$, and $x_n\in C_s$ if $v_n = v_s$; analogously for $y_n$). For each $n\geq 1$, set $P_n:=\{v_n\}$ if $v_n\in V_t$; if $v_n = v_s$ for some $s\in \mathrm{succ}(t)$, then $y_{n-1}, x_n\in C_s$, and the connectedness of $C_s$ (componental property) provides a path $P_n\subseteq C_s$ from $y_{n-1}$ to $x_n$. The concatenation
    \[
        R \;:=\; e_0\,P_1\,e_1\,P_2\,e_2\,\dots
    \]
    is then a ray in $G$ contained in $C_t$ and satisfying $\{e_n\colon n\geq 0\}\subseteq E(R)$.

    Set $\omega := [R]_E \in \Omega_E(G)$. Since $(T,\mathcal{V})$ displays all edge-ends of $G$ and $R\subseteq C_t$, the ray $R_\omega$ in $T$ with $\varphi_{\mathcal{T}}(\omega) = [R_\omega]_E$ passes through $t$. Let $s$ be the successor of $t$ along $R_\omega$. Then $\omega$ lives in $C_s$, so $R$ has a tail in $\overline{C_s}$ and, by finite adhesion, all but finitely many edges of $R$ lie in $E(C_s)$. On the other hand, each $e_n$ has its two endpoints in distinct blocks of the partition $C_t = V_t\sqcup \bigsqcup_{s'\in \mathrm{succ}(t)}C_{s'}$; in particular, $e_n\notin E(C_{s'})$ for any $s'\in\mathrm{succ}(t)$. Since $R$ contains infinitely many edges $e_n$, none of which lie in $E(C_s)$, this contradicts the previous assertion.

    Hence $\tilde{C}_t$ is finite. Moreover, $V_t\subseteq V(\tilde{C}_t)$ is finite and $|\mathrm{succ}(t)| = |\{v_s\colon s\in \mathrm{succ}(t)\}|\leq |V(\tilde{C}_t)|$ is also finite. Together with the edge to $t$'s parent (if any), this gives $t$ finite degree in $T$. As $t\in T$ was arbitrary, $T$ is locally finite and each bag in $\mathcal{V}$ is finite.
\end{proof}

\section{Boundary linked regions and submodularity}
\label{sec:boundarylinked}

In this section, we begin developing the tools needed for the proof of our main result, Theorem~\ref{Main}.
The following lemma shall be often applied throughout this paper and reads as an extension of a previous Lemma 10 from \cite{BruhnStein} regarding locally-finite graphs:

 \begin{lemma}\label{CharacterizationBoundaryLinked}
     A proper region $C$ in a connected graph $G$ is boundary-linked to some undominated edge-end $\omega \in \Omega_E(G)$ if, and only if, we have $|\delta(D)|\geq |\delta(C)|$ for every other region $D$ included in $C$ and in which $\omega$ lives.
 \end{lemma}
 \begin{proof}
     If $C$ is boundary-linked to $\omega$, then there is a family $\mathcal{R}\subseteq \omega$ of $|\delta(C)|$ many edge-disjoint rays in $\overline{C}$ starting at the edges from $\delta(C)$. If $\omega$ lives in another region $D$ included in $C$, then each ray of $\mathcal{R}$ must intersect $\delta(D)$ as well, since $\delta(D)$ separates $\omega$ and the vertices from $N(C)$. Hence, $|\delta(D)|\geq |\mathcal{R}| = |\delta(C)|$ because the elements of $\mathcal{R}$ are pairwise edge-disjoint.

     Conversely, suppose that $|\delta(D)|\geq |\delta(C)|$ for every region $D$ included in $C$ and in which $\omega$ lives. Incidentally, as verified by $\delta(C)$ itself, there is a finite set of edges $F_1\subseteq E(\overline{C})$ separating the vertices of $S_1:=N(C)$ from $\omega$. Choosing $|F_1|$ to be minimum with this property, we can even write $F_1=\delta(C_1)$ for the connected component $C_1$ of $C\setminus F_1$ where $\omega$ lives. In particular, we still have $|\delta(C_1)|\geq |\delta(C)|$ by hypothesis. Setting $C_0:= C$, fix $n\in\mathbb{N}$ and, for every $1\leq i\leq n$, suppose already defined a region $C_i\subseteq C_{i-1}$ in which $\omega$ lives but such that $F_i:=\delta(C_i)$ is a sizewise minimum set of edges separating $\omega$ from some previous $S_i\subseteq C_{i-1}$. Then, denote $S_{n+1}:=\{v\in C_n: v\text{ is an endpoint of an edge in }\delta(C_n)\}$ and let $F_{n+1}\subseteq E(\overline{C_n})$ be also sizewise minimum for separating $\omega$ from the vertices of $S_{n+1}$. In particular, once the elements of $S_{n+1}$ do not edge-dominate $\omega$ by hypothesis, $F_{n+1}$ is indeed finite. Furthermore, the minimality of $|F_{n+1}|$ allows us again to write $F_{n+1}=\delta(C_{n+1})$ for the connected component $C_{n+1}$ of $C_n\setminus F_{n+1}$ where $\omega$ lives. In its turn, since $F_{n+1}$ also separates $\omega$ from $S_n$, we have $|F_{n+1}|\geq |F_n|$ by the sizewise minimality defining $F_n$. Similarly, the choice of $S_{n+1}$ implies that $\delta(C_{n+1})\cap \delta(C_n) = \emptyset$ and, then, that $\delta(C_{n+1})\cap \delta(C_i)=\emptyset$ for every $i<n$.

     At the end of this recursive process, we shall have $\bigcap_{n\in\mathbb{N}}C_n = \emptyset$: otherwise, since $G$ is connected, there would exist a finite path $P$ connecting some $v\in \bigcap_{n\in\mathbb{N}}C_n$ to some $u\in S_1$. In this case, $E(P)$ should intersect $\delta(C_n)$ for every $n\geq 1$, because $v\in C_n\subseteq C$ and $S_1\subseteq V(G)\setminus C$. This would then contradict the fact that $\{\delta(C_n)\}_{n\in\mathbb{N}}$ comprises infinitely many pairwise disjoint sets of edges. Hence, when defining $V_{t_{0}}:=V(G)\setminus C$ and $V_{t_n}:=C_{n-1}\setminus C_{n}$ for every $n\geq 1$, the ray $R:=t_0t_1t_2\dots $ fits as an index tree for a componental tree-cut decomposition $\mathcal{V}:=\{V_{t}:t\in R\}$ of finite adhesion for $G$. In its turn, given a pair $1\leq n<m$, we also have $|\delta(C_{t_n})| = |\delta(C_{n-1})| = \min \{|\delta(C_{i-1})|: n \leq i \leq m\} = \min\{|\delta(C_{t_i})|: n \leq i \leq m\}$. Now, the sizewise minimality of $F_{n-1} = \delta(C_{n-1})$ for separating $S_{n-1}$ from $\omega$ ensures that $F_{n-1}$ has  minimum size while separating $S_{n-1}$ and $C_{m-1}$ as well. Indeed, since $C_{m-1}$ is connected, contains representatives of $\omega$ and is disjoint from $S_{n-1}$ by construction, any set of edges separating $S_{n-1}$ from $C_{m-1}$ also separates it from $\omega$. To summarize, by Menger's theorem there are $|\delta(C_{t_n})|$ many edge-disjoint paths in $G$ connecting the edges of $\delta(C_{t_n}) = \delta(V(G)\setminus C_{n-1})$ to those of $\delta(C_{t_m}) = \delta(C_{m-1})$. Hence, $\mathcal{V}:=\{V_{t}: t\in R\}$ describes a linked and componental tree-cut decomposition for $G$ displaying the edge-end $\omega$, which is thus the only one living in $C_{t_n}$ for every $n\geq 1$. Finishing the proof, Lemma \ref{LinkedDisplay} claims that $C = C_0 = C_{t_1}$ is boundary-linked to $\omega$.    
 \end{proof}
 
 We will not strictly need the following corollary for our later constructions, but it is perhaps still interesting.
 
 \begin{corollary}[\cite{BruhnStein}, Lemma 10]\label{MengerOneEnd}
     For every undominated edge-end $\omega \in \Omega_E(G)$ in a connected graph $G$ and every non-empty finite set of vertices $S\subseteq V(G)$, the maximum size of a family $\mathcal{R}\subseteq \omega$ of edge-disjoint rays starting at $S$ is equal to the minimum size of a finite set of edges $F\subseteq E(G)$ separating $S$ from $\omega$.
 \end{corollary}
 \begin{proof}
    Once the rays from $\mathcal{R}$ are pairwise edge-disjoint and must intersect $F$, we clearly have $|\mathcal{R}|\leq |F|$. On the other hand, by the above Lemma \ref{CharacterizationBoundaryLinked}, the choice of $F$ implies that $C:=C_E(F,\omega)$ is boundary-linked to $\omega$ and that $\delta(C) = F$. Hence, there is a family $\mathcal{R}'\subseteq \omega$ comprising $|\delta(C)| = |F|$ many edge-disjoint rays in $\overline{C}$ starting at the edges from $F$. Similarly, again relying on Menger's theorem, the sizewise minimality of $F$ also provides a family $\mathcal{P}$ of $|F|$ many edge-disjoint paths in $G$ connecting $C$ to the vertices of $S$. Of course, we may ask to each $P\in \mathcal{P}$ intersect $S$ and $C$ precisely in its endpoints. Hence, as every edge $e\in F$ lies on exactly one path in $\mathcal{P}$ and one ray in $\mathcal{R}'$, we can extract from $\bigcup\mathcal{P}\cup \bigcup\mathcal{R}'$ a family of $|F|$ many edge-disjoint rays in $G$ belonging to $\omega$ and starting at the vertices of $S$. Therefore, $|F| = |\mathcal{R}|$ by the sizewise maximality of $\mathcal{R}$.   
 \end{proof}

\begin{lemma}[Uncrossing lemma]
\label{SubmodularityCuts}
    Let $C$ be a proper region of $G$, let $\mathcal{D}$ be a family of pairwise disjoint regions contained in $C$, each boundary-linked to some undominated edge-end, and let $K\subseteq C$ be a region boundary-linked to an undominated edge-end $\omega_K$ that lives in no $D\in\mathcal{D}$. Then there is a region $C'\subseteq C$ (possibly $C'=K$) that is boundary-linked to $\omega_K$, nested with every $D\in\mathcal{D}$, and satisfies $|\delta(C')|\leq|\delta(K)|$.
\end{lemma}

\begin{proof}[Revisited proof of Lemmas 7.5 and 7.6 from \cite{TreeCutPartition}]
If $\delta(K) = \emptyset$, then $K$ is a connected component of $G$ and every region from $\mathcal{D}$ is either contained in $K$ or disjoint from $K$. 
Hence, we may assume $\delta(K)\neq \emptyset$, and prove the lemma by induction on $|\delta(K)|$.
Next, observe that if a given $D\in \mathcal{D}$ is not nested with $K$, then the edge set of $D$ must intersect $\delta(K)$ by the connectedness of $D$. 
Since $\mathcal{D}$ consists of pairwise disjoint regions, the finiteness of $\delta(K)$ implies that $\mathcal{D}':=\{D\in \mathcal{D}: D\text{ is not nested with }K\}$ is also finite. If $\mathcal{D}'=\emptyset$, then $K$ is already nested with all regions in $\mathcal{D}$ and we may take $C' = K$.
We proceed by an inner induction on $|\mathcal{D}'|$. 

Since $K$ is boundary-linked to $\omega_K$, we first let $\mathcal{R}\subseteq\omega_K$ denote an edge-disjoint family of $|\delta(K)|$ many rays in $\overline{K}$ starting at the edges from $\delta(K)$. Now, we fix some $D\in \mathcal{D}'$ and recall that a simple double counting gives the following  `submodularity of cuts': 

\begin{equation}\label{eq:submodularity}
    |\delta(K)|+|\delta(D)|\geq \max\{|\delta(K\cap D)|+|\delta(K\cup D)|, |\delta(K\setminus D)|+|\delta(D\setminus K)|\}.
\end{equation}

Because $\omega_K$ lives in $K\setminus D$ by hypothesis, each ray from $\mathcal{R}$ intersects $\delta(K\setminus D)$ in a different edge, implying $|\delta(K\setminus D)|\geq |\delta(K)|$. This leads to the following case distinctions:

\begin{case}\label{case1}
    Suppose first that $|\delta(K\setminus D)| = |\delta(K)|$. In this case, the connected component $K'$ of $K\setminus D$ in which $\omega_K$ lives is disjoint from $D$ and satisfies $|\delta(K')|\leq |\delta(K\setminus D)| = |\delta(K)|$. Assume now that $K'$ is boundary-linked to $\omega_K$, as illustrated by Figure \ref{fig:SubcaseCase1}. Since $\mathcal{D}$ is a disjoint family, every region of $\mathcal{D}$ not nested with $K'$ belongs to $\mathcal{D}'\setminus\{D\}$. Hence the result follows by applying the inner induction assumption to $K'$. 
    Finally, if $K'$ is not boundary-linked to $\omega_K$, Lemma \ref{CharacterizationBoundaryLinked} ensures the existence of a third region $K'' \subseteq K'$ containing representatives of $\omega_K$, as illustrated by Figure \ref{fig:Subcase2}, and for which $|\delta(K'')|<|\delta(K')|\leq |\delta(K)|$. If we choose $K''$ such that $|\delta(K'')|$ is minimal, then $K''$ becomes boundary-linked to $\omega_K$ (again by Lemma \ref{CharacterizationBoundaryLinked}), and the result follows by applying the outer induction assumption to $K''$.
\end{case}

\begin{figure}[ht]
    \centering

    \begin{subfigure}{0.45\textwidth}
        \centering
        \includegraphics[width=0.7\linewidth]{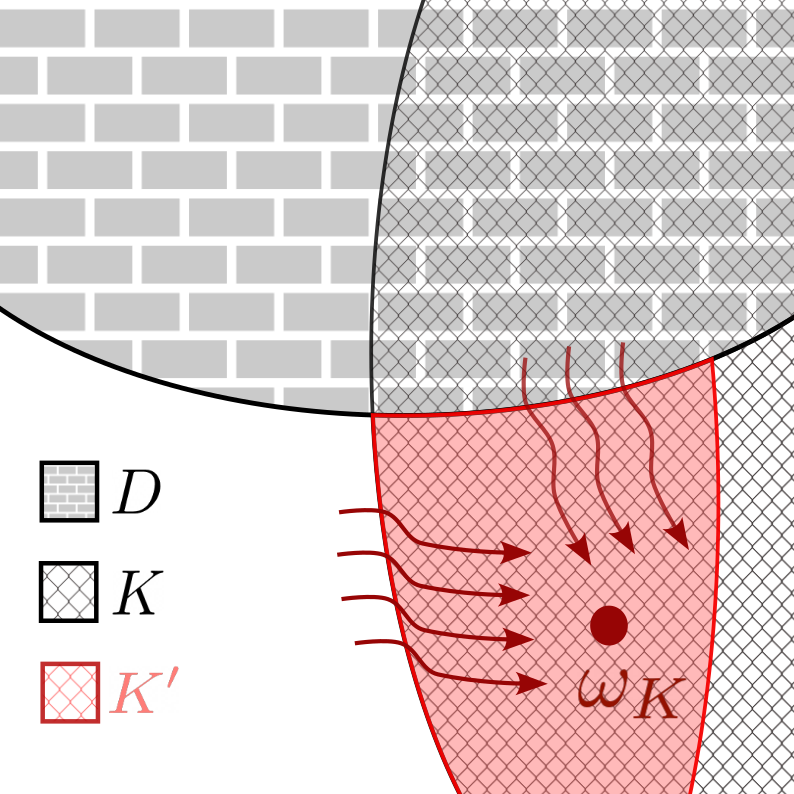}
        \caption{Case in which the connected component $K'$ of $K\setminus D$ is boundary-linked to $\omega_K$.}
        \label{fig:SubcaseCase1}
    \end{subfigure}
    \hfill
    \begin{subfigure}{0.45\textwidth}
        \centering
        \includegraphics[width=0.7\linewidth]{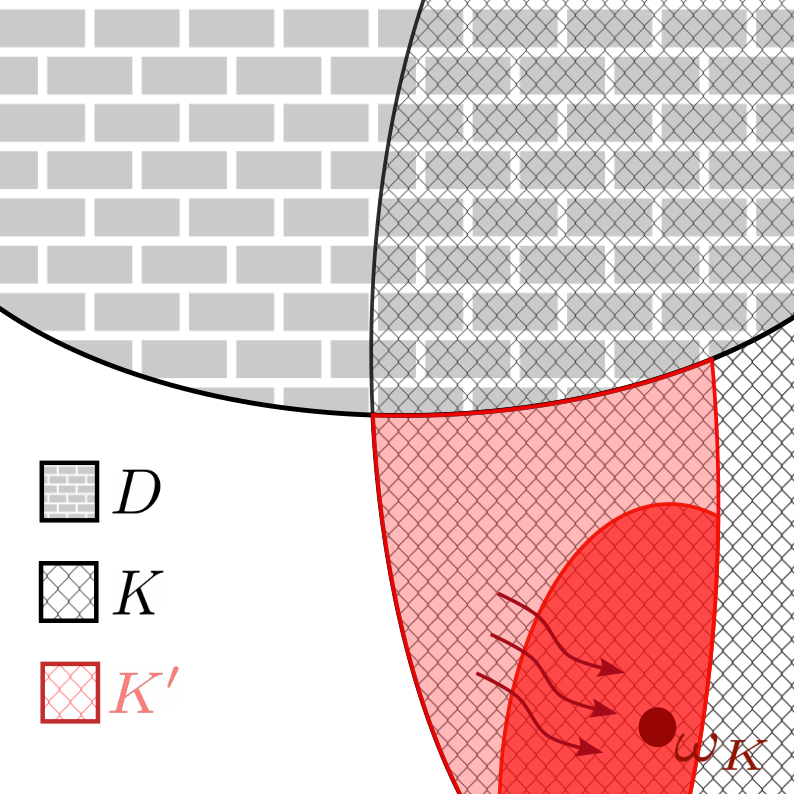}
        \caption{Case in which $\omega_K$ can be separated from the boundary of $K'$ by fewer than $|\delta(K')|$ edges.}
        \label{fig:Subcase2}
    \end{subfigure}

    \caption{Description of Case \ref{case1}, in which $|\delta(K\setminus D)| = |\delta(K)|$.}
    \label{fig:Case1}
\end{figure}

\begin{case}\label{case2}
    Suppose now that $|\delta(K\setminus D)|>|\delta(K)|$. In this case, we have $|\delta(D\setminus K)|<|\delta(D)|$ by (\ref{eq:submodularity}). 
Therefore, the edge-end $\omega_D$ to which $D$ is boundary-linked cannot live in any component of $D\setminus K$, so must belong to $K\cap D$, implying that $|\delta(K\cap D)|\geq |\delta(D)|$. Then it follows from (\ref{eq:submodularity}) that $K':=K\cup D$ satisfies $|\delta(K')|\leq |\delta(K)|$.

Moreover, since $K\cap D\neq\emptyset$, $K'$ is connected. Hence, it remains to show that $K'$ is boundary-linked to $\omega_K$. Since every region of $\mathcal{D}$ not nested with $K'$ belongs to $\mathcal{D}'\setminus\{D\}$, the result will then follow by applying the inner induction assumption to $K'$. 

We now construct a ray family that witnesses that $K'$ is boundary-linked to $\omega_K$. 
First, we partition $\delta(K') = F_1 \sqcup F_2$ where $F_1:= \delta (K')  \setminus \delta(K) \subseteq \delta(D)$ and $F_2 = \delta(K') \cap \delta (K) \subseteq \delta(K)$.
Then let $\mathcal{R}'\subseteq \omega_D$ be a ray family witnessing that $D$ is boundary-linked to $\omega_D$ and keep from this family only the rays starting in $F_1$ (then disregarding the rest). Figure \ref{fig:RayFamilies} illustrates both ray families $\mathcal{R}$ and $\mathcal{R}'$.
Since $\delta(K)$ separates $F_1$ from $\omega_D$, each $Q\in \mathcal{R}'$ intersects $\delta(K)$: let $e_Q \in \delta(K)$ be the first such edge on $Q$, and write $P_Q$ for the initial segment of $Q$ before $e_Q$. Since $e_Q \in E(D)$, we know that $e_Q \notin F_2$.
Let $R_Q\in \mathcal{R}$ denote the unique ray starting at the edge $e_Q$; and write $\mathcal{R}_2 \subseteq \mathcal{R}$ for the set of rays starting in $F_2$. Then, as suggested by Figure \ref{fig:FinalRayFamily}, the family of concatenations $\{{P_Q}^\frown R_Q: Q\in \mathcal{R}'\}\cup \mathcal{R}_2$ consists of pairwise edge-disjoint rays witnessing that $K':=K\cup D$ is boundary-linked to $\omega_K$. 
\end{case}

\end{proof}

\begin{figure}[ht]
    \centering

    \begin{subfigure}{0.45\textwidth}
        \centering
        \includegraphics[width=0.7\linewidth]{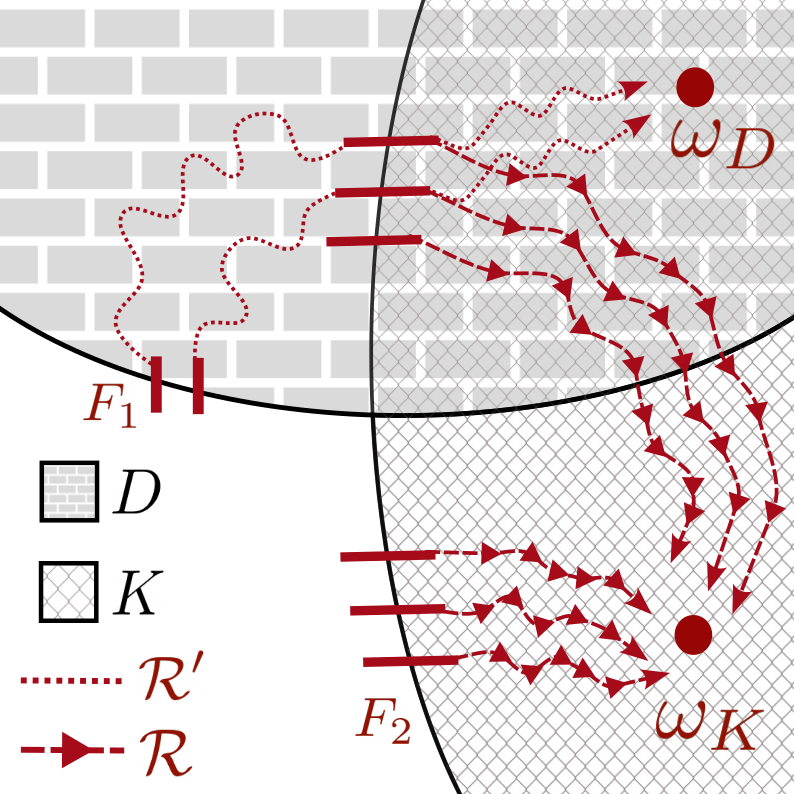}
        \caption{Depiction of the ray families $\mathcal{R}\subseteq \omega_K$ and $\mathcal{R}'\subseteq \omega_D$.}
        \label{fig:RayFamilies}
    \end{subfigure}
    \hfill
    \begin{subfigure}{0.45\textwidth}
        \centering
        \includegraphics[width=0.7\linewidth]{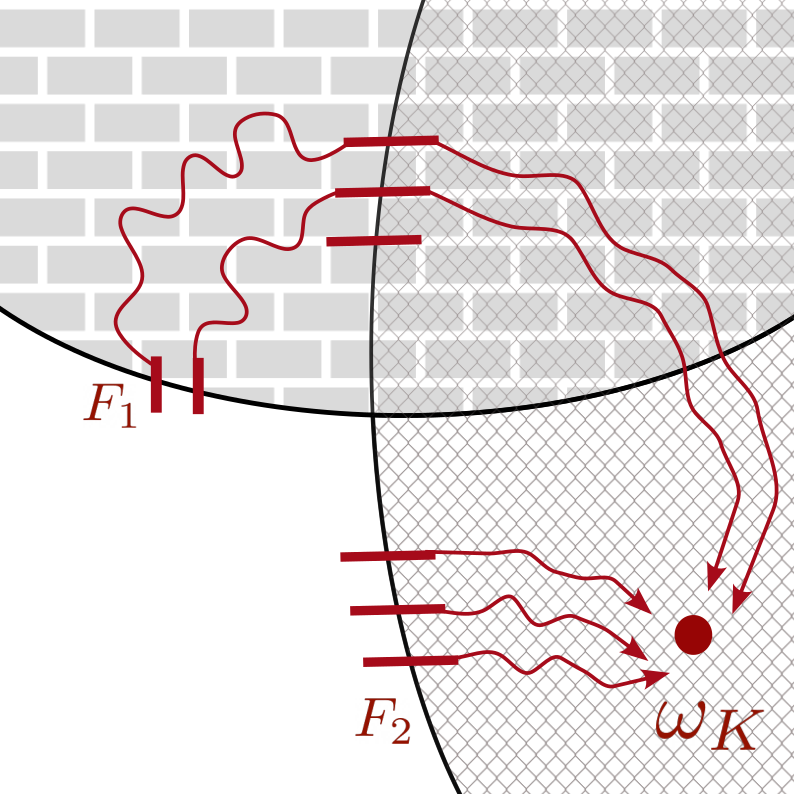}
        \caption{Family of edge-disjoint rays witnessing that $K\cup D$ is boundary-linked to $\omega_K$.}
        \label{fig:FinalRayFamily}
    \end{subfigure}

    \caption{Description of Case \ref{case2}, in which $|\delta(K\setminus D)| > |\delta(K)|$ and, hence, $|\delta(K\cup D)|\leq |\delta(K)|$.} 
    \label{fig:Case2}
\end{figure}

\section{Proof of the main result}\label{sec:proof}

This section proves Theorem \ref{Main} by splitting the construction of the tree-cut decomposition $(T,\mathcal{V})$ into three parts. First, given a region $C$ of the underlying graph $G$, subsection \ref{subsec:algorithm} presents a recursive method for capturing undominated edge-ends that live in $C$ but to which $C$ is not boundary-linked. This approach closely follows the framework developed by Albrechtsen, Jacobs, Knappe and Pitz in \cite{LinkedTreeDecompositions}, centred on their Algorithm 7.2. In fact, comparing with Theorem 7.9 of \cite{LinkedTreeDecompositions}, the first subsection below will establish the following result:

\begin{lemma}\label{LemmaAlgorithm}
    Let $C$ be a region of a graph $G$ and fix a disjoint family $\mathcal{E}$ of boundary-linked regions contained in $C$ with $|\delta(C')|<|\delta(C)|$ for every $C'\in\mathcal{E}$. Then, there is a disjoint family $\Sigma(C,\mathcal{E})$ of boundary-linked regions contained in $C$ satisfying the properties below:
    \begin{enumerate}
        \item \label{Algorithm1} Every $C'\in \mathcal{E}$ is included in some $D\in \Sigma(C,\mathcal{E})$;
        \item\label{Algorithm2} $|\delta(D)|<|\delta(C)|$ for every $D\in \Sigma(C,\mathcal{E})$;
        \item\label{Algorithm3} Let $\omega$ be an undominated edge-end of $G$ living in $C$ but not in any region from $\mathcal{E}$. Suppose further that $\omega$ is an undominated edge-end to which $C$ is not boundary-linked, so that (by Lemma \ref{CharacterizationBoundaryLinked}) it lives in a region $C'\subseteq C$ with $|\delta(C')|<|\delta(C)|$. Then $\omega$ lives in a region $C''\subseteq C$ included in some $D\in \Sigma(C,\mathcal{E})$ and with $|\delta(C'')|\leq |\delta(C')|$;
        \item\label{Algorithm4} Suppose $\mathcal{E}=\emptyset$, and that some region $D'$ of $C$ is such that each $D\in \Sigma(C,\mathcal{E})$ is either disjoint from $D'$ or included in $D'$. Then $|\delta(D)|<|\delta(D')|$ for every $D\in \Sigma(C,\mathcal{E})$ with $D\subsetneq D'$.
    \end{enumerate}
    In particular, $C$ is boundary-linked to every edge-end not living in any $D\in \Sigma(C,\mathcal{E})$.
\end{lemma}

In the case where $\mathcal{E} = \emptyset$, we write $\Sigma(C):=\Sigma(C,\emptyset)$ and call any such family satisfying (\ref{Algorithm1})--(\ref{Algorithm4}) a \textbf{cutting family} for $C$. To complement the above result (especially regarding its third item), subsection \ref{sec:successor} uses Pitz's Theorem \ref{DisplayKurkofka} to partition \emph{all} undominated edge-ends within $C$, including those to which $C$ is itself boundary-linked. The details are recorded in the following technical lemma:

\begin{lemma}\label{LemmaKurkofka}
    Let $C$ be a region of an infinite graph $G$ and suppose that $\mathcal{E}$ is a disjoint family of boundary-linked regions contained in $C$. Assume further that the following condition holds:
    \begin{center}
        $(\dagger)$ If some region $D'$ of $C$ is such that each $C'\in \mathcal{E}$ is either disjoint from $D'$ or included in $D'$, then $|\delta(C')|<|\delta(D')|$ for every $C'\in \mathcal{E}$ with $C'\subsetneq D'$.
    \end{center}

    Then, there is a disjoint family $\mathbf{S}(C,\mathcal{E})$ of boundary-linked regions contained in $C$ satisfying the following properties:
    \begin{enumerate}
        \item\label{Kurkofka1} Every $C'\in \mathcal{E}$ is included in some region from $\mathbf{S}(C,\mathcal{E})$;
        \item\label{Kurkofka2} If $D\in \mathbf{S}(C,\mathcal{E})$ satisfies $\delta(D)\cap \delta(C)\neq \emptyset$, then $D\in \mathcal{E}$;
        \item\label{Kurkofka3} Every undominated edge-end living in $C$ also lives in some $D\in \mathbf{S}(C,\mathcal{E})$.
    \end{enumerate}
\end{lemma}

Comparing the statements of the above two results, we see that condition $(\dagger)$ in Lemma \ref{LemmaKurkofka} coincides with item (\ref{Algorithm4}) of Lemma \ref{LemmaAlgorithm}. Accordingly, we call $\mathbf{S}(C,\mathcal{E})$ a \textbf{successor family} for a region $C$ in a graph $G$ if it is produced by Lemma \ref{LemmaKurkofka} applied with $\mathcal{E}:=\Sigma(C)$ a cutting family. It is worth remarking that some elements of $\mathbf{S}(C,\Sigma(C))$ themselves might be inherited from cutting families of auxiliary regions contained in $C$ but which were not boundary-linked. The first item in the latter result then requires a ``book keeping'' hypothesis from the former one, which suggests why Lemma \ref{LemmaAlgorithm} is also written in terms of a parameter family $\mathcal{E}$. All in all, when indeed constructing the tree-cut decomposition $(T,\mathcal{V})$ for $G$ in subsection \ref{subsec:buildingup}, the members of $\mathbf{S}(C_t,\Sigma(C_t))$ will be precisely the regions $C_s$ (as in subsection \ref{subsec:tree-cut}) with $s\in \mathrm{succ}(t)$.

\subsection{An algorithm for capturing regions of small boundary size}\label{subsec:algorithm} A first step in revisiting the theory of subsections 7.1 and 7.2 of \cite{LinkedTreeDecompositions} is to record some combinatorial properties of regions in a graph $G$. Following Lemma 7.3 of that paper, we say that two sets of vertices $C,D\subseteq V(G)$ are \textbf{nested} if $C\subseteq D$, $D\subseteq C$ or $C\cap D = \emptyset$. Building on this notion, the lemma below is useful for producing $\subseteq$-maximal regions under combinatorial constraints:

\begin{lemma}\label{PropertiesRegions}
    Let $G$ be a connected graph and $\mathcal{C}$ be a chain of regions of $G$, in the sense that $C\subseteq C'$ or $C'\subseteq C$ for every pair $C,C'\in\mathcal{C}$. If there is some $n\in\mathbb{N}$ such that $|\delta(C)|\leq n$ for every $C\in \mathcal{C}$, then:
    \begin{enumerate}
        \item\label{Regions1} $D:=\bigcup \mathcal{C}$ is a region of $G$ with $|\delta(D)|\leq n$;
        \item\label{Regions2'} If every $C\in \mathcal{C}$ is nested with some $X\subseteq V(G)$, then so is $D$;
        \item\label{Regions2} If every region in $\mathcal{C}$ is boundary-linked, then there is some $C\in \mathcal{C}$ such that $D$ is boundary-linked to every edge-end to which $C$ is boundary-linked. 
    \end{enumerate}
\end{lemma}
\begin{proof}
For item (\ref{Regions2'}), the hypothesis that each $C\in\mathcal{C}$ is nested with $X$ ensures that at least one of the following holds: (i) $X\subseteq C$ for some $C\in\mathcal{C}$, whence $X\subseteq D$; (ii) $C\subseteq X$ for every $C\in\mathcal{C}$, whence $D\subseteq X$; or (iii) $C\cap X=\emptyset$ for every $C\in\mathcal{C}$, whence $D\cap X=\emptyset$. Indeed, if (i) fails then each $C\in\mathcal{C}$ satisfies $C\subseteq X$ or $C\cap X=\emptyset$, and the $\subseteq$-chain property forces these alternatives to be shared by every non-empty $C\in\mathcal{C}$. In each case, $D$ is nested with $X$. Connectedness of $D$ follows since any two $C,C'\in\mathcal{C}$ are connected subsets with non-empty intersection. The remaining items rely on the intermediate claim below.

\begin{center}
\textbf{Claim.} For every finite set of edges $F\subseteq E(G)$, each of which has at least one endpoint in $D$, there is some $C\in \mathcal{C}$ with $F\subseteq E(\overline{C})$.
\end{center}

\begin{proof}[Proof of the claim]
Let $S\subseteq D$ consist of the endpoints in $D$ of the edges of $F$, so that each $v\in S$ belongs to some $C_v\in\mathcal{C}$ by definition of $D$. Since $\mathcal{C}$ is a $\subseteq$-chain and $S$ is finite, there is a $\subseteq$-maximum region $C$ in $\{C_v : v\in S\}$. Then $S\subseteq C$, so $F\subseteq E(\overline{C})$.
\end{proof}

For item (\ref{Regions1}), suppose for contradiction that $|\delta(D)|>n$ and pick a finite $F\subseteq\delta(D)$ of size $n+1$. The claim yields $C\in\mathcal{C}$ with $F\subseteq E(\overline{C})$. Since each edge of $F$ has an endpoint outside $D\supseteq C$, we have $F\subseteq\delta(C)$, contradicting $|\delta(C)|\leq n$. Hence, $|\delta(D)|\leq n$ and applying the claim with $F=\delta(D)$ yields some $C\in\mathcal{C}$ with $\delta(D)\subseteq E(\overline{C})$. As above, $\delta(D)\subseteq\delta(C)$.

For item (\ref{Regions2}), fix a $C\in\mathcal{C}$ as produced at the end of the proof of (\ref{Regions1}) (so that $\delta(D)\subseteq\delta(C)$) and let $\omega$ be any edge-end to which $C$ is boundary-linked. Then $\overline{C}$ contains a family of $|\delta(C)|$ edge-disjoint, edge-equivalent rays in $\omega$ whose first edges are precisely the edges of $\delta(C)$. Since $C\subseteq D$ implies $E(\overline{C})\subseteq E(\overline{D})$, these rays also lie in $\overline{D}$. Restricting to the rays whose first edge lies in $\delta(D)\subseteq\delta(C)$ then gives a subfamily of $|\delta(D)|$ edge-disjoint, edge-equivalent rays in $\overline{D}$ whose first edges are precisely the edges of $\delta(D)$. Hence $D$ is boundary-linked to $\omega$.
\end{proof}

In order to address the proof of Lemma \ref{LemmaAlgorithm}, fix from now on in this section a region $C$ of a graph $G$ and a disjoint family $\mathcal{E}$ of boundary-linked regions contained in $C$ with $|\delta(C')|<|\delta(C)|$ for every $C'\in\mathcal{E}$. The construction of $\Sigma(C,\mathcal{E})$ satisfying (\ref{Algorithm1})--(\ref{Algorithm4}) proceeds via an auxiliary family $\Sigma'(C,\mathcal{E}):=\{C_i\}_{i<\Omega}$ extending $\mathcal{E}$. Since we require $\Sigma'(C,\mathcal{E})\supseteq\mathcal{E}$, we begin by fixing an enumeration $\{C_i\}_{i<\alpha}$ of $\mathcal{E}$ for some ordinal $\alpha$. For $i\geq\alpha$, assuming $\{C_j\}_{j<i}$ has already been defined, we call a region $C'\subseteq C$ an $i$-\textbf{candidate} if the following hold:

\begin{itemize}
    \item $|\delta(C')|<|\delta(C)|$;
    \item $C'$ is boundary-linked to an undominated edge-end that lives in $C$ but in no $C_j$ with $j<i$;
    \item $C'$ is nested with $C_j$ for every $j<i$.
\end{itemize}

Together, the last two conditions imply that each $C_j$ with $j<i$ is either disjoint from $C'$ or included in $C'$: indeed, we have $C'\not\subseteq C_j$, since $C'$ is boundary-linked to an undominated edge-end which lives in $C'$ but not in $C_j$. If an $i$-candidate exists, we choose $C_i$ according to the following rule:

\begin{algorithm}[Revisited Algorithm 7.2 from \cite{LinkedTreeDecompositions}]\label{ChoiceCi}
    Among all $i$-candidates, we first choose $C_i$ to minimize $|\delta(C_i)|$, and we further require $C_i$ to be $\subseteq$-maximal within this boundary-size class. Existence of such a $\subseteq$-maximal choice is ensured by Zorn's Lemma: given a $\subseteq$-chain $\mathcal{C}$ of $i$-candidates sharing the common boundary size $n:=\min\{|\delta(C')|:C'\text{ is an }i\text{-candidate}\}$, set $D:=\bigcup\mathcal{C}$. On one hand, items (\ref{Regions1}) and (\ref{Regions2'}) of Lemma \ref{PropertiesRegions} show that $D$ is a region with $|\delta(D)|\leq n$ and nested with every $C_j$, $j<i$. On the other hand, item (\ref{Regions2}) yields some $C'\in\mathcal{C}$ such that $D$ is boundary-linked to every edge-end to which $C'$ is boundary-linked. In particular, $D$ is boundary-linked to an undominated edge-end witnessing the candidacy of $C'$ (hence living in $C$ but in no $C_j$, $j<i$). Thus $D$ is itself an $i$-candidate and the minimality of $n$ forces $|\delta(D)|=n$, so that $D$ is an upper bound for $\mathcal{C}$ with the same boundary size.
\end{algorithm}

The construction terminates at the first ordinal $\Omega$ for which no $\Omega$-candidate exists within $C$. We record the following fact for use in the proof of Lemma \ref{LemmaAlgorithm}:

\begin{lemma}\label{WellDefinitionCi}
    The family $\Sigma'(C,\mathcal{E}):=\{C_i\}_{i<\Omega}$ contains no infinite, strictly $\subseteq$-ascending chain. 
\end{lemma}
\begin{proof}
    Suppose for contradiction that $C_{i_0}\subsetneq C_{i_1}\subsetneq C_{i_2}\subsetneq\cdots$ is such a chain, and set $\mathcal{C}:=\{C_{i_n} \colon n<\omega\}$ and $D:=\bigcup\mathcal{C}$. Since $\mathcal{E}$ is a disjoint family, at most one chain member lies in $\mathcal{E}$. Dropping a finite initial segment if necessary, we may assume $i_n\geq\alpha$ for every $n \in \mathbb{N}$.

    By item (\ref{Regions2}) of Lemma \ref{PropertiesRegions}, there is some $C_i\in\mathcal{C}$ (with $i\geq\alpha$) such that $D$ is boundary-linked to every edge-end to which $C_i$ is boundary-linked. Let $\omega$ be the undominated edge-end witnessing the candidacy of $C_i$. In particular, $\omega$ lives in $C_i$ but in no $C_k$ with $k<i$, and $D$ is boundary-linked to $\omega$. By Lemma \ref{CharacterizationBoundaryLinked}, $|\delta(D)|\leq|\delta(C_i)|<|\delta(C)|$.

    Hence, aiming to verify that $D$ is itself an $i$-candidate, we already concluded that its boundary size is $<|\delta(C)|$ and that it is boundary-linked to $\omega$, which lives in no $C_k$ with $k<i$. Finally, by the pairwise nestedness of $\{C_i\}_{i<\Omega}$, every chain member is nested with $C_k$ for every $k<i$, from where item (\ref{Regions2'}) of Lemma \ref{PropertiesRegions} gives that $D$ is nested with $C_k$ for every $k<i$.

    Since $i > \alpha$, we know by Algorithm \ref{ChoiceCi} that $|\delta(C_i)|$ is the minimum boundary size among $i$-candidates, so $|\delta(C_i)|\leq|\delta(D)|$ and hence $|\delta(D)|=|\delta(C_i)|$. But then the $\subseteq$-maximality of $C_i$ among $i$-candidates of this boundary size contradicts $C_i\subsetneq D$, which holds because $C_i$ is a chain member and the chain is strictly ascending.
\end{proof}

In particular, for every $i\leq\Omega$, each region in $\{C_j\}_{j<i}$ is contained in a $\subseteq$-maximal member of that family: if some $C_k$ were not, one could iteratively select $C_{k_0}:=C_k\subsetneq C_{k_1}\subsetneq C_{k_2}\subsetneq\cdots$ within $\{C_j\}_{j<i}$, contradicting Lemma \ref{WellDefinitionCi}.

Moreover, since the family $\Sigma'(C,\mathcal{E}):=\{C_i\}_{i<\Omega}$ itself consists of pairwise nested regions by construction, the maximal ones from $\Sigma(C,\mathcal{E})$ are actually pairwise disjoint. In fact, each $D\in \Sigma(C,\mathcal{E})$ is an $i-$candidate for some ordinal $i<\Omega$, so that $D$ is boundary-linked to an undominated edge-end and with boundary size $|\delta(D)|<|\delta(C)|$. Then, considering the above statement, our definition for $\Sigma(C,\mathcal{E})$ indeed meets conditions (\ref{Algorithm1}) and (\ref{Algorithm2}) from Lemma \ref{LemmaAlgorithm}.

Finally, relying now both on the above result and on Lemma \ref{CharacterizationBoundaryLinked}, we finish this subsection by arguing how $\Sigma(C,\mathcal{E})$ also meets properties (\ref{Algorithm3}) and (\ref{Algorithm4}) from Lemma \ref{LemmaAlgorithm}:

\begin{proof}[Proof of Lemma \ref{LemmaAlgorithm}]
    First, in order to prove item (\ref{Algorithm3}) from Lemma \ref{LemmaAlgorithm}, let $\omega \in \Omega_E(G)$ be an undominated edge-end living in $C$ but not in any region within $\mathcal{E}$. Suppose further that $\omega$ lives also in a region $C'\subseteq C$ with $|\delta(C')|<|\delta(C)|$. 
    By passing to a subregion of $C'$ if necessary, we may assume by Lemma \ref{CharacterizationBoundaryLinked} that $C'$ is boundary-linked to $\omega$. 
    
    Set $i:=\min\{j<\Omega : \omega\text{ lives in }C_j\}$ if this set is non-empty and $i:=\Omega$ otherwise. Since $\omega$ lives in no region from $\mathcal{E}=\{C_j\}_{j<\alpha}$ by hypothesis, we have $i\geq\alpha$. Note then that the family $\mathcal{D}:=\{\subseteq\text{-maximal regions from }\{C_j\}_{j<i}\}$ is well-defined by Lemma \ref{WellDefinitionCi} and whose elements are regions in which $\omega$ does not live, precisely by the choice of $i$. Applying the Uncrossing Lemma \ref{SubmodularityCuts} to $K=C'$ and $\mathcal{D}$, we obtain a region $\tilde C\subseteq C$ boundary-linked to $\omega$, nested with every $D\in\mathcal{D}$, and with $|\delta(\tilde C)|\leq|\delta(C')|<|\delta(C)|$. Since $\omega$ lives in $\tilde C$ but in no $C_j$ with $j<i$, each $D\in\mathcal{D}$ is either disjoint from $\tilde C$ or included in $\tilde C$. By Lemma \ref{WellDefinitionCi}, $\tilde C$ is therefore nested with every $C_j$, $j<i$, and is thus an $i$-candidate.

    Due to the definition of $\Omega$ as the ordinal in which Algorithm \ref{ChoiceCi} is finished, we must have $i<\Omega$. Then the choice of $C_i$ for minimizing $|\delta(C_i)|$ as in Algorithm \ref{ChoiceCi} ensures that $|\delta(C_i)|\leq|\delta(\tilde C)|\leq|\delta(C')|$, as well as that $C_i$ itself satisfies the requirements for $C''$ in item (\ref{Algorithm3}): $\omega$ lives in $C_i$ by the choice of $i$, $C_i$ is contained in some region of $\Sigma(C,\mathcal{E})$ by Lemma \ref{WellDefinitionCi}, and $|\delta(C_i)|\leq|\delta(C')|$ as just shown.

\medskip
    Finally, for proving item (\ref{Algorithm4}) from Lemma \ref{LemmaAlgorithm}, suppose $\mathcal{E} = \emptyset$ and fix a region $D'$ of $C$ such that each $D\in \Sigma(C,\mathcal{E})$ is either disjoint from $D'$ or included in $D'$. 
     Then, $D'$ is nested with all elements from $\Sigma'(C,\mathcal{E}) = \{C_j \colon j < \Omega\}$.
     Now consider $D\in \Sigma(C,\mathcal{E})$ with $D\subsetneq D'$. Let $i<\Omega$ be the index for which $D = C_i$, and let $\omega$ be an edge-end to which $C_i$ is boundary-linked but that lives in no $C_j$ with $j<i$. We show $|\delta(D)|<|\delta(D')|$ by distinguishing whether $D'$ is boundary-linked to $\omega$ or not.

     If $D'$ is boundary-linked to $\omega$ too, then $D\cap D'\neq \emptyset$ and $D'$ is itself an $i$-candidate, so the minimality of $|\delta(C_i)|$ from Algorithm \ref{ChoiceCi} gives $|\delta(D')|\geq |\delta(C_i)| = |\delta(D)|$. Equality would contradict the $\subseteq-$maximality of $C_i$ among $i$-candidates of this boundary size, since $C_i \subsetneq D'$. This proves $|\delta(D)|<|\delta(D')|$, as required.

     Otherwise, by Lemma \ref{CharacterizationBoundaryLinked} there exists a region $C'\subseteq D'$ which is boundary-linked to $\omega$ and satisfies $|\delta(C')|<|\delta(D')|$. We claim that $|\delta(C')|\geq |\delta(D)|$: otherwise, Lemma \ref{SubmodularityCuts} would ensure the existence of some $C''\subseteq C$ still boundary-linked to $\omega$, also nested with the $\subseteq-$maximal elements from $\{C_j\}_{j<i}$ and with $|\delta(C'')|\leq |\delta(C')|<|\delta(D)|$. So $C''$ would be an $i$-candidate, contradicting the minimality of $|\delta(C_i)|$ as in Algorithm \ref{ChoiceCi}. Combining the inequalities gives $|\delta(D)|\leq |\delta(C')|<|\delta(D')|$, completing the proof of item (\ref{Algorithm4}).
     \end{proof}

\subsection{Partitioning the undominated edge-end structure}\label{sec:successor}
We now address the construction of a disjoint family $\mathbf{S}(C,\mathcal{E})$ of boundary-linked regions as in Lemma \ref{LemmaKurkofka}. 
To that aim, from now on throughout this subsection, let $C$ be a fixed region of $G$ and $\mathcal{E}$ be a prescribed collection of pairwise disjoint, boundary-linked regions included in $C$.

Denote by $\hat{C}$ the (multi-)graph obtained from $C$ after contracting each $D\in \mathcal{E}$ to a vertex $v_D\in V(\hat{C})$. Note that each such vertex $v_D$ will have finite degree $|\delta(D)|$.
Hence, the vertices of infinite degree in $\hat{C}$ are actually also vertices of infinite degree in $G$. 
Furthermore, we may regard the edges from $E(\hat{C})\cup \delta(C)$ as a subset of $E(\overline{C})$, so that the finite set $X:=\{v\in V(\hat{C}) \colon v\text{ is an endpoint of an edge from }\delta(C)\}$ is well-defined. 

Now let $\mathcal{V} = \{V_t:t\in T\}$ be a fixed tree-cut decomposition of finite adhesion for $\hat{C}$ as in Theorem \ref{DisplayKurkofka}. 
Then, the set of nodes $S:=\{t\in T : V_t\cap X \neq \emptyset\}$ is also finite and every connected component $T'$ of $T\setminus S$ defines a region of $T$. Moreover, the same Theorem \ref{DisplayKurkofka} implies that $\hat{C}_{T'}:=\bigcup_{t\in T'}V_t$ is a region of $\hat{C}$ as well. 
Since every $D\in \mathcal{E}$ is connected, we may uncontract all vertices from $\{v \in V(\hat{C}_{T'}): v=v_D\text{ for some }D\in \mathcal{E}\}$ in order to obtain a region $C_{T'}$ of $C$. 

Furthermore, when $T''$ and $T'$ are distinct connected components of $T \setminus S$, then $\hat{C}_{T''}\cap \hat{C}_{T'} = \emptyset$ and thus $C_{T''}\cap C_{T'} = \emptyset$.

    Now, since $\mathcal{E}$ is a disjoint family and $X\subseteq \bigcup_{t\in S}V_t$, we note that the elements of the auxiliary family $\mathbf{S}'(C,\mathcal{E}):=\{D\in \mathcal{E} \colon v_D\in X\}\cup \{C_{T'} \colon T'\text{ is a connected component of }T\setminus S\}$ are also pairwise disjoint. Due to $\mathcal{V}$ being a near partition of $\hat{C}$, for each $D\in \mathcal{E}$ we either have $D\in \mathbf{S}'(C,\mathcal{E})$ (if $v_D \in X$) or $D\subseteq C_{T'}$ for some connected component $T'$ of $T\setminus S$ (if $v_D\in V_t$ for some $t\in T'$). In the latter case, we note that $\hat{C}_{T'}$ contains no endpoint of an edge from $\delta(C)$ because $T'\cap S=\emptyset$ and, hence, $X\cap \hat{C}_{T'} =  X\cap \bigcup_{t\in T'}V_t = \emptyset$. This proves the first two items of the following lemma:

    \begin{lemma}\label{AuxiliarySuccessorFamily}
        The family $\mathbf{S}'(C,\mathcal{E})$ consists of pairwise disjoint regions and satisfies the three properties below:
        \begin{enumerate}
            \item\label{Auxiliary1} Every $C'\in \mathcal{E}$ is included in some region from $\mathbf{S}'(C,\mathcal{E})$;
            \item\label{Auxiliary2} If $D\in \mathbf{S}'(C,\mathcal{E})$ satisfies $\delta(D)\cap \delta(C) \neq \emptyset$, then $D\in\mathcal{E}$;
            \item\label{Auxiliary3} Every undominated edge-end living in $C$ lives also in some $D\in \mathbf{S}'(C,\mathcal{E})$.
        \end{enumerate}
    \end{lemma}
    \begin{proof}
        It only remains to prove item (\ref{Auxiliary3}). 
        For this, let $R$ be a ray in $C$ that belongs to an undominated edge-end $\omega \in \Omega_E(G)$. Due to item (\ref{Auxiliary1}) already proven, we may assume that $\omega$ lives in no region from $\mathcal{E}$. In this case, $R\cap D$ is finite for every $D\in \mathcal{E}$ because the cut $\delta(D)$ is itself finite. For the same reason, the set of edges $E(R)\cap E(\hat{C})$ then forms a connected and locally finite subgraph of $\hat{C}$, so K\H{o}nig's Lemma yields a ray $\hat{R}$ in $\hat{C}$ with $E(\hat{R}) \subseteq E(R)$.

We claim that $\hat{R}$ belongs to an undominated edge-end $\hat{\omega}$ of $\hat{C}$.
       Indeed, we already saw that every vertex $v$ of infinite degree in $\hat{C}$ is a vertex of $C$, and hence does not dominate $\omega$. 
       In particular, there is a finite set of edges $F\subseteq E(C)$ separating $v$ from $\omega$ in $C$. 
       Therefore, after considering $\mathcal{E}_F:=\{D\in \mathcal{E} \colon F\cap E(D) \neq \emptyset\}$, the finite set of edges 
       $$F':=(F\cap E(\hat{C}))\cup\displaystyle \bigcup_{\substack{D\in \mathcal{E}_F}}\delta(D)$$
        separates $v$ from $\hat{R}$ in $\hat{C}$. Indeed, any path in $\hat{C}\setminus F'$ from $v$ to a tail of $\hat{R}$ would lift to a path in $C$ by replacing each traversal of a vertex $v_D$ with a path inside $D$ (which exists because $D$ is connected), and this lift would use no edge of $F$. 
        Since $\mathcal{E}_F$ and $F$ are finite and $\hat{R}$ is a ray, we may choose a tail of $\hat{R}$ that avoids both the finitely many vertices $\{v_D \colon D\in\mathcal{E}_F\}$ and the finitely many edges of $F\cap E(\hat{C})$. Similarly, extending along such a tail in the same fashion would connect $v$ to a tail of $R$ in $G\setminus F$, contradicting that $F$ separates $v$ from $\omega$. Hence, $\hat{R}$ belongs to an undominated edge-end of $\hat{C}$ as well.
       
      Therefore,  Theorem \ref{DisplayKurkofka} ensures that this edge-end $\hat{\omega}$ is displayed by some ray $R_{\hat{\omega}}$ of $T$.
      As $S$ is finite, $R_{\hat{\omega}}$ has a tail in some connected component $T'$ of $T\setminus S$, 
      while the ray $\hat{R}$ thus has a tail in the corresponding region $\hat{C}_{T'}$ of $\hat{C}$. 
      In particular, infinitely many edges from $E(\hat{R})\subseteq E(R)$ must have both endpoints in $C_{T'}\in \mathbf{S}'(C,\mathcal{E})$, concluding that $\omega$ lives in this region $C_{T'}$ of $C$ because $\delta(C_{T'})$ is finite.      
    \end{proof}

Comparing the above statement with Lemma \ref{LemmaKurkofka}, we see that the family $\mathbf{S}'(C,\mathcal{E})$ falls short of the requirements of Lemma \ref{LemmaKurkofka} as its members are not necessarily boundary-linked regions. 
Nevertheless, relying on the previous Lemma \ref{LemmaAlgorithm}, we shall finish the current subsection by arguing that this condition can be verified once we assume $\mathcal{E}$ to also satisfy $(\dagger)$:

\begin{proof}[Proof of Lemma \ref{LemmaKurkofka}]
    Suppose that $\mathcal{E}$ satisfies property $(\dagger)$ from Lemma \ref{LemmaKurkofka}, which we recall below for convenience:
    \begin{center}
        $(\dagger)$ If some region $D'$ of $C$ is such that each $C'\in \mathcal{E}$ is either disjoint from $D'$ or included in $D'$, then $|\delta(C')|<|\delta(D')|$ for every $C'\in \mathcal{E}$ with $C'\subsetneq D'$.
    \end{center}

    Let $\mathbf{U}$ consist of all regions from $\mathbf{S}'(C,\mathcal{E})$ that are not boundary-linked.
    Given $D\in \mathbf{U}$, we define $\mathcal{E}_D:=\{D'\in \mathcal{E} : D'\subseteq D\}$. As a subset of $\mathcal{E}$, this is a disjoint family of boundary-linked regions. Since every member of $\mathcal{E}$ is boundary-linked but $D\in\mathbf{U}$ is not, $D\notin\mathcal{E}$ and, thus, $D'\subsetneq D$ for every $D'\in \mathcal{E}_D$. On the other hand, item (\ref{Auxiliary1}) of Lemma \ref{AuxiliarySuccessorFamily} places each $D'\in \mathcal{E}\setminus \mathcal{E}_D$ inside some element of $\mathbf{S}'(C,\mathcal{E})\setminus\{D\}$, which is disjoint from $D$. In particular, $D'\cap D=\emptyset$ as well. Thus, each element of $\mathcal{E}$ is either disjoint from $D$ or included in $D$, so property $(\dagger)$ applies and yields $|\delta(D')|<|\delta(D)|$ for every $D'\in \mathcal{E}_D$. We are therefore under the hypotheses of Lemma \ref{LemmaAlgorithm}, which furnishes a disjoint family $\Sigma(D,\mathcal{E}_D)$ of boundary-linked regions satisfying (\ref{Algorithm1})--(\ref{Algorithm4}).

    Set
    \[
        \mathbf{S}(C,\mathcal{E}):=\bigl(\mathbf{S}'(C,\mathcal{E})\setminus \mathbf{U}\bigr)\cup \bigcup_{D\in\mathbf{U}}\Sigma(D,\mathcal{E}_D).
    \]
    Members of $\mathbf{S}'(C,\mathcal{E})\setminus\mathbf{U}$ are boundary-linked regions of $C$ by definition of $\mathbf{U}$ and by Lemma \ref{AuxiliarySuccessorFamily}, while members of each $\Sigma(D,\mathcal{E}_D)$ are boundary-linked regions of $D\subseteq C$ by Lemma \ref{LemmaAlgorithm}. 
    As members of $\mathbf{S}'(C,\mathcal{E})$ are pairwise disjoint by Lemma \ref{AuxiliarySuccessorFamily}, and each $\Sigma(D,\mathcal{E}_D)$ individually consists of pairwise disjoint regions included in $D$ by Lemma \ref{LemmaAlgorithm}, we know that all members of $\mathbf{S}(C,\mathcal{E})$ are pairwise disjoint.
    
    We now verify that $\mathbf{S}(C,\mathcal{E})$ satisfies all properties required in Lemma \ref{LemmaKurkofka}:
    
    For \emph{Item (\ref{Kurkofka1})}, fix $C'\in \mathcal{E}$. By item (\ref{Auxiliary1}) of Lemma \ref{AuxiliarySuccessorFamily}, $C'\subseteq D_0$ for some $D_0\in \mathbf{S}'(C,\mathcal{E})$. If $D_0\notin\mathbf{U}$, then $D_0\in \mathbf{S}(C,\mathcal{E})$ and we are done. Otherwise, $D_0\in\mathbf{U}$ and $C'\in \mathcal{E}_{D_0}$, so item (\ref{Algorithm1}) of Lemma \ref{LemmaAlgorithm} supplies some $D''\in \Sigma(D_0,\mathcal{E}_{D_0})\subseteq\mathbf{S}(C,\mathcal{E})$ with $C'\subseteq D''$.

    For \emph{Item (\ref{Kurkofka2})}, suppose that $D'\in \mathbf{S}(C,\mathcal{E})$ satisfies $\delta(D')\cap\delta(C)\neq\emptyset$. If $D'\in \mathbf{S}'(C,\mathcal{E})\setminus\mathbf{U}$, item (\ref{Auxiliary2}) of Lemma \ref{AuxiliarySuccessorFamily} gives $D'\in \mathcal{E}$. If instead $D'\in \Sigma(D,\mathcal{E}_D)$ for some $D\in\mathbf{U}$, then the chain $D'\subseteq D\subseteq C$ shows that any edge in $\delta(D')\cap\delta(C)$ also lies in $\delta(D)$, so $\delta(D)\cap\delta(C)\neq\emptyset$. Item (\ref{Auxiliary2}) then yields $D\in \mathcal{E}$ --- contradicting that $\mathcal{E}$-members are boundary-linked while $D\in\mathbf{U}$ is not. This second case therefore does not occur.

    For \emph{Item (\ref{Kurkofka3})}, let $\omega$ be an undominated edge-end living in $C$. By item (\ref{Auxiliary3}) of Lemma \ref{AuxiliarySuccessorFamily}, $\omega$ lives in some $D_0\in \mathbf{S}'(C,\mathcal{E})$. If $D_0\notin\mathbf{U}$, then $D_0\in \mathbf{S}(C,\mathcal{E})$ and we are done. We have $D_0\in\mathbf{U}$ otherwise, so $D_0$ is boundary-linked to no undominated edge-end --- in particular not to $\omega$ --- and the final assertion of Lemma \ref{LemmaAlgorithm} then places $\omega$ in some member of $\Sigma(D_0,\mathcal{E}_{D_0})\subseteq\mathbf{S}(C,\mathcal{E})$.
\end{proof}

\subsection{Building up a linked tree-cut decomposition }\label{subsec:buildingup}

We are now ready to construct the linked tree-cut decomposition $(T,\mathcal{V})$ for a connected graph $G$ as claimed by our main Theorem \ref{Main}. 
The desired tree-cut decomposition will arise as the limit of a sequence $\{(T_n,\mathcal{V}_n)\}_{n\in\mathbb{N}}$ of tree-cut decompositions of $G$, each of which satisfies:
\begin{itemize}
    \item[(i)] $T_n$ is a rayless tree whose leaves $t\in L(T_n):=\{\text{leaves of }T_n\}$, for $n\geq 1$, correspond to boundary-linked regions $V_t^n$ of $G$;
    \item[(ii)] every undominated edge-end $\omega\in\Omega_E(G)$ lives in $V_t^n$ for some leaf $t\in L(T_n)$.
\end{itemize}
To start the inductive procedure, we let $(T_0,\mathcal{V}_0)$ be the trivial tree-cut decomposition of $G$ where $T_0:=\{t_0\}$ and $V_{t_0}^0:=V(G)$. For constructing $(T_1,\mathcal{V}_1)$, we fix a successor family $\mathbf{S}_0:=\mathrm{S}(G,\emptyset)$ as provided by Lemma \ref{LemmaKurkofka} when considering $\mathcal{E} = \emptyset$. Let $T_1:=\{t_0\}\cup\mathbf{S}_{0}$ be the tree with root $t_0$ and one neighbour of the root for every $D \in \mathbf{S}_{0}$.
    Define $V_{t_0}^1:=V(G)\setminus\bigcup\mathbf{S}_{0}$ and $V_{D}^1:=D$ for every $D\in\mathbf{S}_{0}$. 
    By item (\ref{Kurkofka3}) in Lemma \ref{LemmaKurkofka}, $\mathcal{V}_1=\{V_t^1 \colon t\in T_1\}$ is then a tree-cut decomposition of $G$ satisfying the above items (i) and (ii).

We now assume by induction that $(T_n,\mathcal{V}_n)$ satisfying (i) and (ii) has been defined for some $n\geq 1$.
For every leaf $t\in L(T_n)$, Lemma \ref{LemmaAlgorithm} yields a cutting family $\mathcal{E}_t:=\Sigma(V_t^n)$, and Lemma \ref{LemmaKurkofka} yields a successor family $\mathbf{S}_t:=\mathbf{S}(V_t^n,\mathcal{E}_t)$ for $V_t^n$.
Similar to the previous paragraph, consider the extension of $T_n$ given by $T_{n+1}:=T_n\cup \bigcup_{t\in L(T_n)}\mathbf{S}_t$, where for every $t\in L(T_n)$ the family $\mathbf{S}_t$ comprises precisely the successors of $t$ in $T_{n+1}$.
Setting 
\begin{itemize}
\item $V_D^{n+1}:=D$ for every $t \in L(T_n)$ and $D\in \mathbf{S}_t$,
\item $V_t^{n+1}:=V_t^{n}\setminus \bigcup \mathbf{S}_t$ for every $t \in L(T_n)$, and 
\item $V_s^{n+1}:=V_s^n$ if $s\in T_n\setminus L(T_n)$, 
\end{itemize}
we obtain a tree-cut decomposition $(T_{n+1},\mathcal{V}_{n+1})$.
Again due to item (\ref{Kurkofka3}) in Lemma \ref{LemmaKurkofka}, each undominated edge-end of $G$ living in $V_t^n$ for some $t\in L(T_n)$ must now live in $V_{s}^{n+1}$ for some successor $s$ of $t$ in $T_{n+1}$. 
This completes the inductive step.

\medskip
At the end of this recursive process, fix the limit tree $T:=\bigcup_{n\in\mathbb{N}} T_n$ and set $V_t:=V_t^{n_t+1}$ for every $t\in T$. Adopting the convention $L(T_0) := \{t_0\}$, for every node $t\in T$ the index $n_t:=\min\{n\in\mathbb{N}: t\in T_n\}$ is the unique one for which $t\in L(T_{n_t})$. Then, if $s,t\in T$ are distinct, $V_t$ and $V_s$ are also disjoint bags from the tree-cut decomposition $\mathcal{V}_{n_s+n_t+1}$. Moreover, the region $V_t^{n_t}$  of $G$ can now be alternatively described as follows:

\begin{lemma}\label{WellDefinitionCt}
    For every $t \in T$, we have $V_t^{n_t} = \bigcup_{s\geq t}V_s $. 
\end{lemma}
\begin{proof}
    First, fix a pair $s,t \in T$ with $s\geq t$ and, for some $k\in\mathbb{N}$, denote by $t_1t_2t_3\dots t_k$ the unique path in $T$ with $t = t_1$ and $s=t_k$. In this case, $t_{i+1}$ is a successor of $t_i$ for every $1\leq i < k$, whence $n_{t_{i+1}} = n_{t_i}+1$ and $V_{t_{i+1}}^{n_{t_{i+1}}}\in \mathbf{S}_{t_i}$ by construction.
    Then, the inclusions $V_{s} = V_{t_k} \subseteq V_{t_k}^{n_{t_k}}\subseteq V_{t_{k-1}}^{n_{t_{k-1}}}\subseteq \dots \subseteq V_{t_1}^{n_{t_1}}$ immediately imply that $V_s\subseteq V_t^{n_t}$, giving the inclusion $\bigcup_{s\geq t}V_s \subseteq V_{t}^{n_t}$.

 To conclude the reverse inclusion, suppose for a contradiction that there exists some `uncovered' vertex $v\in V_t^{n_t}\setminus \bigcup_{s\geq t}V_s$. Let $s_0=t$, and suppose for some $k\in\mathbb{N}_0$, we have already defined a path $s_0s_1s_2s_3\dots s_k$ in $T$ such that $s_0<s_1<s_2<\dots<s_k$ and $v\in V_{s_i}^{n_{s_i}}$.
 Since $v\in V_{s_k}^{n_{s_k}}\setminus V_{s_k}$ and $\mathbf{S}_{s_k}$ is a partition of $V_{s_k}^{n_{s_k}}\setminus V_{s_k}$, there must indeed exist a successor $s_{k+1}$ of $s_k$ such that $v\in V_{s_{k+1}}^{n_{s_{k}}+1} = V_{s_{k+1}}^{n_{s_{k+1}}}$. 
 Thus, this recursive process results in a ray $s_0s_1s_2\ldots$ in $T$.
  
Writing $C_k:=V_{s_k}^{n_{s_k}}$, we thus obtain a sequence $\{C_k\}_{k\in\mathbb{N}}$ of regions of $G$ such that $C_{k+1}\in \mathbf{S}_{s_k} = \mathbf{S}(C_k,\Sigma(C_k))$ and $v\in C_k $ for every $k\in\mathbb{N}$. In particular, we have $C_0\supseteq C_1\supseteq \dots$ and $v\in \bigcap_{k\in\mathbb{N}}C_k$.

    However, since $G$ is connected, there is some finite path $P$ from $v$ to $v_0\in V_{t_0}$. 
    Since $v_0\notin C_k$ for $k\geq 1$, we have $E(P)\cap \delta(C_k)\neq \emptyset$ for every $k\in\mathbb{N}_{\geq 1}$. 
    Since the path $P$ is finite, some edge $e\in E(P)$ must belong to all but finitely many $\delta(C_k)$. 
    Hence, we have $\delta(C_k)\cap \delta(C_{k+1})\neq \emptyset$ for all sufficiently large $k \in \mathbb{N}$. 
    But then item (\ref{Kurkofka2}) in Lemma \ref{LemmaKurkofka} ensures that $C_{k+1}\in \Sigma(C_k)$. 
    This, however, is a contradiction, as now the inequality $|\delta(C_{k+1})|<|\delta(C_k)|$ holds for every such $k$ by item (\ref{Algorithm2}) in Lemma \ref{LemmaAlgorithm}. 
\end{proof}

Following the notation from Subsection \ref{subsec:tree-cut}, Lemma \ref{WellDefinitionCt} allows us to define $C_t:=\bigcup_{s\geq t}V_s = V_t^{n_t}$ as a region of $G$ for each $t\in T$. Moreover, by construction of $(T_{n_t+1},\mathcal{V}_{n_t+1})$, the successor family $\mathbf{S}_t = \{V_s^{n_t+1} : s \text{ is a successor of } t \text{ in } T_{n_t+1}\}$ can alternatively be written as $\{C_s : s \text{ is a successor of } t \text{ in } T\}$ for every $t\in T$. We conclude Section \ref{sec:proof} by verifying that $(T,\mathcal{V})$ has all the properties claimed in Theorem \ref{Main}, thus finishing its proof.

\begin{claim}
    $(T,\mathcal{V})$ is a componental tree-cut decomposition of $G$ of finite adhesion.
\end{claim}
\begin{proof}[Proof of the claim]
    Applying Lemma \ref{WellDefinitionCt} at $t=t_0$ with the convention $V_{t_0}^0 := V(G)$, we have $V(G) = V_{t_0}^0 = \bigcup_{s\in T} V_s$, so $\mathcal{V}$ covers $V(G)$. Together with the disjointness noted above Lemma \ref{WellDefinitionCt}, this makes $\mathcal{V}=\{V_t:t\in T\}$ a partition of $V(G)$.

    In addition, for every edge $e=st\in E(T)$ with $s<t$, the adhesion set $F_e$ can also be written as $\delta\left(\bigcup_{s'\geq t}V_{s'}\right) = \delta(C_t)$, which is finite since $C_t$ is a region of $G$.

    Finally, the region $C_t$ is connected for every $t\in T$, and boundary-linked (hence non-empty) for every $t\neq t_0$.
\end{proof}

\begin{claim}
    $(T,\mathcal{V})$ displays the undominated edge-ends of $G$.
\end{claim}
\begin{proof}
    To verify this, we need to check that every undominated edge-end of $G$ lives in some end of $T$ and that, conversely, every end of $T$ hosts exactly one edge-end of $G$, which is moreover undominated.

We begin with the first direction. Let $\omega \in \Omega_E(G)$ be undominated. By item (\ref{Kurkofka3}) in Lemma \ref{LemmaKurkofka}, $\omega$ lives in $C_{t_1}\in \mathbf{S}_{t_0}$ for some successor $t_1$ of $t_0$. Inductively, if a finite path $t_0t_1\dots t_n$ in $T$ has been defined for which $\omega$ lives in $C_{t_n}$, then since $\mathbf{S}_{t_n}=\{C_s:s\text{ is a successor of }t_n\}$ is a successor family, applying item (\ref{Kurkofka3}) again yields $\omega$ living in $C_{t_{n+1}}$ for some successor $t_{n+1}$ of $t_n$. The resulting ray $R_\omega:=t_0t_1t_2\dots$ in $T$ then satisfies $\varphi_{\mathcal{T}}(\omega) = [R_\omega]_E\in\Omega(T)$, so $\omega$ lives in the end $[R_\omega]_E$ of $T$.

Turning to the converse, fix an end of $T$ represented by a ray $R:=t_0t_1t_2\dots$ at $t_0$. We first construct an edge-end of $G$ that lives in $\varphi_{\mathcal{T}}^{-1}([R]_E)$. Since each $C_{t_n}$ is a non-empty region of $G$ by the componental property verified previously, and since $\bigcap_{n\in\mathbb{N}} C_{t_n} = \emptyset$ (because $\mathcal{V}$ is a partition of $V(G)$), we may inductively extract a subsequence $\{t_{n_k}\}_{k\in\mathbb{N}}$ by choosing $n_{k+1}>n_k$ so that $C_{k+1}:=C_{t_{n_{k+1}}}$ avoids some previously fixed $v_k\in C_k:=C_{t_{n_k}}$. Each such $C_{k+1}$ is then non-empty and strictly contained in $C_k$. No vertex $v\in V(G)$ can be connected to infinitely many of the $v_k$ by a system of infinitely many edge-disjoint paths: otherwise, since $v\notin C_k$ for all sufficiently large $k$ while $C_k$ contains all but finitely many $v_i$, we would contradict the finiteness of $\delta(C_k)$. The Star-Comb Lemma (see Lemma 8.2.2 in \cite{DiestelBook}) hence yields a ray $S\subseteq G$ connected to infinitely many $v_k$ by pairwise edge-disjoint paths. Combined with the finiteness of each $\delta(C_k)$ and the fact that $C_k$ contains all but finitely many $v_i$, this forces a tail of $S$ in every $C_k$, and therefore in every $C_{t_n}$. In other words, $\varphi_{\mathcal{T}}([S]_E) = [R]_E$.

Moreover, as no vertex $v\in V(G)$ can be connected to infinitely many of the $v_k$, it follows that this edge-end is undominated.

To check that $[S]_E$ is the only edge-end of $G$ mapping to $[R]_E$, let $S'$ be a second ray in $G$ with $\varphi_{\mathcal{T}}([S']_E)=[R]_E$, so that $S'$ also has a tail in every $C_{t_n}$. For any finite $F\subseteq E(G)$, fix $n$ large enough that $C_{t_n}$ contains no endpoint of an edge from $F$. The connectedness of $C_{t_n}$ (as ensured by the componental property) then places tails of $S$ and $S'$ in the same connected component of $G\setminus F$. Hence $S$ and $S'$ are edge-equivalent, i.e., $[S]_E=[S']_E$.

\end{proof}

\begin{claim}
    $(T,\mathcal{V})$ is linked.
\end{claim}
\begin{proof}
    Let $t_1t_2\dots t_n$ be a finite path in $T$ with $t_1<t_2<\dots<t_n$. Since $C_{t_{i+1}}$ belongs to the successor family $\mathbf{S}_{t_i}$ for every $1\leq i<n$, each $C_i:=C_{t_i}$ is boundary-linked. Hence, let $\omega_i$ denote an undominated edge-end to which $C_i$ is boundary-linked.
Suppose for a contradiction that there is no family of $\kappa:=\min\{|\delta(C_{i})|: 1 \leq i \leq n\}$ many edge-disjoint paths in $G$ connecting the edges from $\delta(C_{1})$ to those of $\delta(C_{n})$.
Then, by Menger's theorem, there is a set of edges $F\subseteq E(G)$ of size $|F|<\kappa$ separating the disjoint subgraphs $C_n$ and $V(G)\setminus C_1$.
Assuming $F$ of minimum size for this separation property, we have $F\subseteq E(\overline{C_1})\setminus E(C_n)$ and $F = \delta(C)$ for the connected component $C$ of $G[C_1]-F$ including $C_n$.

Let $1\leq i \leq n$ be the maximum index for which $C_i$ includes a region $C'$ where $\omega_n$ lives but with $|\delta(C')|<\kappa$.
Since $C_n$ is boundary-linked to $\omega_n$ and $\kappa \leq |\delta(C_n)|$ by definition, $i<n$.
On the other hand, by item (\ref{Algorithm3}) of Lemma \ref{LemmaAlgorithm} and because $|\delta(C')|<\kappa \leq |\delta(C_i)|$, the edge-end $\omega_n$ also lives in some region $C''\subseteq C_i$ with $|\delta(C'')|\leq |\delta(C')|<\kappa$ that is included in some $D'\in \mathcal{E}_{t_i} = \Sigma(C_i)$.
Therefore, item (\ref{Kurkofka1}) from Lemma \ref{LemmaKurkofka} provides some $D\in \mathbf{S}_{t_i} = \mathbf{S}(C_i, \mathcal{E}_{t_i})$ including $D'$. However, $C_{i+1}$ also belongs to $\mathbf{S}_{t_i}$ by construction, and since both $C_{i+1}$ and $D$ contain tails of rays of $\omega_n$, they intersect.
As $\mathbf{S}_{t_i}$ is a disjoint family of regions within $C_i$, we must have $C_{i+1} = D \supseteq C''$. Then, recalling that $\omega_n$ lives in $C''$ with $|\delta(C'')|<\kappa$, we contradict the maximality of $i$.
\end{proof}

\bibliographystyle{plain}
\bibliography{references}

\appendix
\section{Counterexample regarding the strongly-linked property}\label{section:appendix}

Within finite graph theory, the existence of linked tree-cut decompositions attaining optimal width is the core of the mentioned work due to Giannopoulou, Kwon, Raymond and Thilikos in \cite{GIANNOPOULOU20211}. However, again reflecting the treatment due to Thomas in \cite{THOMAS199067} for finite tree-width, the corresponding index trees in their discussions are not required to be rooted and, hence, suggest an even more general statement to the linkedness condition. More precisely, we say that a tree-cut decomposition $(T,\mathcal{V})$ for a graph $G$ is \textbf{strongly-linked} if, for every pair of edges $e,e'\in E(T)$, there are $\min\{|F_{f}|:f\in E(eTe')\}$ edge-disjoint paths in $G$ connecting the edges from the adhesion set $F_e$ to those of $F_{e'}$, where $eTe'$ denotes the unique sizewise minimum path in $T$ containing $e$ and $e'$. Just as the paper \cite{LinkedTreeDecompositions} by Albrechtsen, Jacobs, Knappe and Pitz inspires Theorem \ref{Main}, their counterexamples in \cite{CounterexamplesLinkedTreeDecompositions} with regard to strongly-linked tree-decompositions suggest that our main result also cannot be stated while considering this ``unrooted version'' of linkedness.
Indeed, the aim of this appendix is to construct a locally finite graph as in the following result:

\begin{theorem}\label{CounterexampleStronglyLinkedness}
    There is a locally-finite graph $G$ such that no tree-cut decomposition displaying all its ends is strongly-linked.
\end{theorem}

In order to introduce useful notation, the two remarks below first describe auxiliary multigraphs for building $G$ up:

\begin{remark}\label{constructionDelta}
    For a given pair of vertices $u,v$ and a natural number $n\in\mathbb{N}$, let $\Delta(u,v,n)$ be a multigraph whose vertex set is $\{u,v\}\sqcup \{x_k\}_{k\in\mathbb{N}}$ and which contains the following systems of parallel edges: there are $n$ edges connecting $u$ to $x_0$, as well as $n$ edges connecting $v$ to $x_0$ and, for each $k\in\mathbb{N}$, also $2n$ edges connecting $x_k$ to $x_{k+1}$. In particular, as also suggested by Figure \ref{fig:Delta3} below, $\Delta(u,v,n)$ has a unique edge-end.
\end{remark}

\begin{figure}[h]
    \centering
    \begin{tikzpicture}[scale=1.5]

  \def\nsteps{6}    
  \def\xoffset{2}   
  \def\s{1}         
  \def\nparUV{3}    
  \def\nparRay{6}

\node[circle,fill,inner sep=1.5pt,label=left:$x_0$] (x0) at (\xoffset,0) {};

  \node[circle,fill,inner sep=1.5pt,label=above:$u$] (u)
    at ($(x0)+(-\s/2,{sqrt(3)/2*\s})$) {};
  \node[circle,fill,inner sep=1.5pt,label=below:$v$] (v)
    at ($(x0)+(-\s/2,{-sqrt(3)/2*\s})$) {};

  \foreach \i in {1,...,\nparUV} {
    \pgfmathsetmacro{\bend}{6*(\i-(\nparUV+1)/2)}
    \draw[bend left=\bend] (u) to (x0);
  }

  \foreach \i in {1,...,\nparUV} {
    \pgfmathsetmacro{\bend}{6*(\i-(\nparUV+1)/2)}
    \draw[bend right=\bend] (v) to (x0);
  }

  \foreach \i in {1,...,\nsteps} {
    \node[circle,fill,inner sep=1.5pt,label=below:$x_{\i}$] (x\i)
      at ($(x0)+(\i*\s,0)$) {};
  }

  \foreach \i in {0,...,\numexpr\nsteps-1} {
    \pgfmathtruncatemacro{\k}{\i+1}
    \foreach \j in {1,...,\nparRay} {
      \pgfmathsetmacro{\bend}{10*(\j-(\nparRay+1)/2)} 
      \draw[bend left=\bend] (x\i) to (x\k);
    }
  }

  \draw[->,thick] (x\nsteps) -- ++(0.8,0);

\end{tikzpicture}
    \caption{Construction of $\Delta(u,v,3)$.}
    \label{fig:Delta3}
\end{figure}
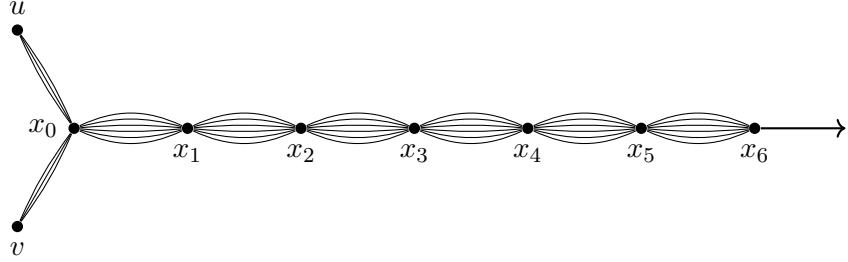

\begin{remark}\label{constructionH}
    Let $H'$ be the following infinite ladder: it is obtained from two disjoint rays $R_v:=v_0v_1v_2\dots$ and $R_u:=u_0u_1u_2\dots$ after also adding the set of edges $\{u_nv_n\}_{n\geq 1}$. Then, for each $n\geq 1$, replace the edges $v_nv_{n+1}$ and $u_{n}u_{n+1}$ by $n+1$ parallel edges connecting the corresponding pairs of endpoints $\{v_n,v_{n+1}\}$ and $\{u_n,u_{n+1}\}$. Let $H$ denote the resulting multigraph as in Figure \ref{fig:constructionH}, whose unique edge-end $\varepsilon$ has then infinite edge-degree. 
\end{remark}

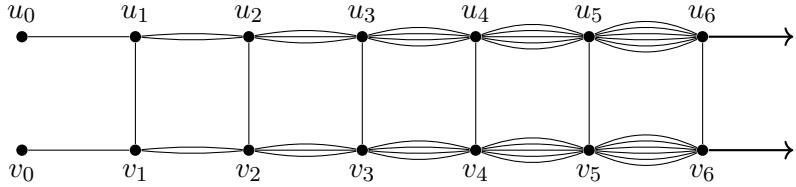
\begin{figure}[h]
\centering
\begin{tikzpicture}[scale=1.5, every node/.style={circle,fill,inner sep=1.5pt}]

  \def\nsteps{6}

  \foreach \i in {0,...,\nsteps} {
    \node[fill] (u\i) at (\i,1) {};
    \node[fill] (v\i) at (\i,0) {};
    \node[draw=none,fill=none] at (\i,1.2) {$u_{\i}$}; 
    \node[draw=none,fill=none] at (\i,-0.2) {$v_{\i}$}; 
  }

  \foreach \i in {0,...,\numexpr\nsteps-1} {
    \pgfmathtruncatemacro{\m}{\i+1}     
    \pgfmathtruncatemacro{\k}{\i+1}    
    \foreach \j in {1,...,\m} {
      \pgfmathsetmacro{\bend}{10*(\j-(\m+1)/2)}
      \draw[bend left=\bend] (u\i) to (u\k);
      \draw[bend left=\bend] (v\i) to (v\k);
    }
    \ifnum\i>0
      \draw (u\i) -- (v\i);
    \fi
  }

  \draw (u\nsteps) -- (v\nsteps);

  \draw[->,thick] (u\nsteps) -- ++(0.8,0);
  \draw[->,thick] (v\nsteps) -- ++(0.8,0);

\end{tikzpicture}

\caption{Construction of the multigraph $H$, where for each $n\in\mathbb{N}$ there are $n+1$ parallel edges between the pairs of vertices $\{v_n,v_{n+1}\}$ and $\{u_n,u_{n+1}\}$.}
\label{fig:constructionH}
\end{figure}

Considering the above description of $H$, let $G$ be the locally-finite multigraph defined by $G := H\cup \bigcup_{n\geq 1}\Delta(u_{2n-1},u_{2n},2n+1)$. Then, given $n\geq 1$, write $\omega_n$ for the unique (undominated edge-)end living in the subgraph $\Delta_n:=\Delta(u_{2n-1},u_{2n},2n+1)$. In addition, let $F_n\subseteq E(G)$ comprise the $2n+1$ parallel edges connecting the pair $\{u_{2n},u_{2n+1}\}$ together with the $2n$ edges of the form $\{u_iv_i: 1 \leq i \leq 2n\}$. Hence, $|F_{n}| = 4n+1$ and we can also alternatively write $F_{n} = \delta\left(\{u_0\}\cup\bigcup_{1 \leq k \leq n}\Delta_k\right)$, from where it follows that $C_n:= C_E(F_n,\omega_n)$ and $C_n':=C_E(F_n,\omega)$ are disjoint regions of $G\setminus F_n$. Indeed, we have $C_E(F_n,\omega_n) = \{u_0\}\cup \bigcup_{1\leq k \leq n}\Delta_k$ and $C_{E}(F_n,\omega) = V(G)\setminus C_E(F_n,\omega_n)$, because every endpoint in $V(G)\setminus C_E(F_n,\omega_n)$ of an edge from $F_n$ lies on a tail of $R_u$ or $R_v$ (as in Remark \ref{constructionH}) not intersecting $F_n$. Intuitively, the result below then checks the sizewise minimality of $F_n$ for separating $\omega_n$ from $\omega$ in $G$:

\begin{lemma}
    $D_n:=C_n\setminus \{u_0\}$ and $D_n':=C_n'\setminus \{v_0\}$ are regions boundary-linked to $\omega_n$ and $\omega$ respectively.
\end{lemma}
\begin{proof}
    Since both vertices $u_0$ and $v_0$ have degree $1$ in $G$, it clearly follows that $D_n$ and $D_n'$ as above indeed induce connected subgraphs of $G$ having finite boundary. In fact, $|\delta(D_n)| = |\delta(D_n')| = 4n+2$ since $ \delta(D_n)\setminus \{u_0u_1\} = F_n= \delta(D_n')\setminus \{v_0v_1\}$ and $|F_n| = 4n+1$. In its turn, fix the sequence $\{x_i^n\}_{i\in\mathbb{N}} := \Delta_n\setminus \{u_{2n},u_{2n-1}\}$ as in the construction of $\Delta_n$ through Remark \ref{constructionDelta}. Now, let $D\subseteq D_n$ be a region in which $\omega_n$ lives. If $D\subseteq \Delta_n$, then we can write $D=\{x_i\}_{i\geq i_0}$ for some $i_0\in\mathbb{N}$ and conclude that $|\delta(D)| = 4n+2$: after all, $\delta(D) = E(x_{0},u_{2n-1})\cup E(x_0,u_{2n})$ if $i_0=0$ and $\delta(D)=E(x_{i_0-1},x_{i_0})$ otherwise. On the other hand, if $D\not\subseteq \Delta_k$, the indexes $k_{\mathrm{min}}:=\min\{1\leq k \leq 2n : u_k\in D\}$ and $k_{\mathrm{max}}:=\max\{1 \leq k \leq 2n: u_k\in D \}$ are well-defined. In particular, due to the connectedness of $D$, we have $u_k \in D$ for every $k_{\mathrm{min}}\leq k \leq k_{\mathrm{max}}$. By the same argument and because $\omega_n$ lives in $D_n$, it also follows that $2n-1\leq k_{\mathrm{max}}\leq 2n$. If $k_{\mathrm{max}} = 2n-1$, then we have $E(u_{2n-1},u_{2n})\cup E(x_0,u_{2n})\subseteq D$ again by construction of $\Delta_n$ as in Remark \ref{constructionDelta} and, in particular, 
    \begin{align*}
        |\delta(D)| & \geq |E(u_{2n-1},u_{2n})|+|E(u_{2n-1},x_0)|+ |\{u_kv_k:k_{\mathrm{min}}\leq k\leq k_{\mathrm{max}}\}|  \\
        & = 2n + (2n+1)+(k_{\max}-k_{\min}+1) \\
        & \geq 4n+2.
    \end{align*}
    If $k_{\max} = 2n$, we similarly have $E(u_{2n},u_{2n+1})\subseteq \delta(D)$ because $D\subseteq D_n$, from where the inequality below now holds:
    \begin{align*}
        |\delta(D)| & \geq |E(u_{2n},u_{2n+1})| +  |\{u_kv_k:k_{\mathrm{min}}\leq k\leq k_{\mathrm{max}}\}| + |E(u_{k_{\min}-1},u_{k_{\min}})|  \\
        & = (2n + 1) +(k_{\max}-k_{\min}+1) + k_{\min} \\
        & = 4n+2.
    \end{align*}
    In any case, $D_n$ is boundary-linked to $\omega_n$ by Lemma \ref{CharacterizationBoundaryLinked}, since we just showed that $|\delta(D)|\geq 4n+2 = |\delta(D_n)|$ for every region $D\subseteq D_n$ in which $\omega_n$ lives.

    On the other hand, suppose now that $D\subseteq D_n'$ is a region in which $\omega$ lives. Once both rays $R_v=v_0v_1v_2\dots$ and $R_u=u_0u_1u_2\dots$ belong to $\omega$ by construction as in Remark \ref{constructionH}, now the indexes $k_1:=\min\{k\in\mathbb{N}:v_k \in D\}$ and $k_2:=\min\{k\in\mathbb{N}: u_k\in D\}$ are well-defined. Incidentally, we have $k_1 \geq 2n+1$ because $\{u_0,u_1,u_2,\dots,u_{2n}\}\subseteq C_n$, as well as $k_2\geq 1$ since $v_0\notin D_n'\supseteq D$. If $k_2\geq 2n+1$ too, then it even follows that $|\delta(D)|\geq |E(v_{k_1-1},v_{k_1})| +  |E(u_{k_2-1},u_{k_2})| = k_1+k_2 \geq (2n+1)+(2n+1) = 4n+2$. If $k_2<2n+1\leq k_1$, on the other hand, then 
    \begin{align*}
        |\delta(D)| & \geq |E(v_{k_1-1},v_{k_1})| +  |\{u_kv_k:k_{2}\leq k< k_{1}\}| + |E(u_{k_{2}-1},u_{k_{2}})|  \\
        & = k_1 +(k_{1}-k_{2}) + k_{2} \\
        & = 2k_1\\
        & \geq 4n+2.
    \end{align*}
    Therefore, as in the previous analysis, we just concluded the inequality $|\delta(D)|\geq 4n+2 = |\delta(D_n')|$ for every region $D\subseteq D_n'$ in which $\omega$ lives. Equivalently, according to Lemma \ref{CharacterizationBoundaryLinked}, $D_n'$ is boundary-linked to $\omega$.
\end{proof}

Then, considering $D_n$ and $D_n'$ as in the above result, let $\mathcal{R}_n$ and $\mathcal{R}_n'$ be two families of $4n+2$ edge-disjoint rays in $\overline{D_n}$ and $\overline{D_n'}$, respectively, starting at the edges from $\delta(D_n) = \delta(C_n)\cup \{u_0u_1\} = F_n\cup \{u_0u_1\}$ and $\delta(D_n') = \delta(C_n')\cup \{u_0u_1\} = F_n\cup \{u_0u_1\}$. If $G$ admits a strongly-linked tree-cut decomposition displaying $\omega$ and $\omega_n$, these ray systems are useful for witnessing suitable adhesion sets:

\begin{lemma}\label{nthSeparator}
    If $G$ admits a linked tree-cut decomposition of finite adhesion $(T,\mathcal{V})$ displaying $\omega$ and $\omega_n$, then there is $f\in E(T)$ whose corresponding adhesion set $F_f$ has size $4n+1$ and separates $u_0$ from $v_0$.
\end{lemma}
\begin{proof}
    In fact, let $t_0$ denote a root for $T$ and consider the corresponding tree-order $\leq$. Since $(T,\mathcal{V})$ displays $\omega_n$, there must exist a ray $R:=t_0t_1t_2\dots$ in $T$ such that $\omega_n$ lives in $C_{t_k}:=\bigcup_{s\geq t_k}V_{s}$ for every $k\in\mathbb{N}$. Similarly, there is a ray $R':=t_0't_1't_2'\dots$ starting at $t_0' = t_0$ and whose subgraphs of the form $C_{t_i'}:=\bigcup_{s\geq t_i'}V_s$, $i\in\mathbb{N}$, contain representatives of $\omega$. Because $F_n$ is finite, we can also choose large enough $i\in\mathbb{N}$ so that neither $E(C_{t_i})$ nor $E(C_{t_{i}'})$ intersects $F_n$. Hence, the connected component $K_i$ of $C_{t_i}$ in which $\omega_n$ lives must be included in $C_n$, as well as $K_i'\subseteq C_n'$ for the connected component $K_i'$ of $C_{t_i'}$ where $\omega$ has its representatives.

    In particular, since $\mathcal{R}_n\subseteq \omega_n$, each ray $R_e\in \mathcal{R}_n$ starting at a given edge $e\in F_n\subseteq \delta(C_n)\subseteq \delta(D_n)$ includes a path $P_e$ in $\overline{C_n}$ containing $e$ and intersecting $\delta(K_i)\subseteq \delta(C_{t_i})$. Analogously, because $\mathcal{R}_n'\subseteq \omega$, from each ray $R_e'\in \mathcal{R}_n'$ starting at $e$ we can extract a finite path $P_e'$ in $\overline{C_n'}$ containing $e$ and some edge of $\delta(K_i')\subseteq \delta(C_{t_i}')$. Observing that $P_e$ and $P_e'$ meet precisely at $e$, the family of edge-disjoint paths $\mathcal{P}:=\{P_eP_e': e \in F_n\}$ connects the adhesion sets of $(T,\mathcal{V})$ described by $\delta(C_{t_i})$ and $\delta(C_{t_i'})$.

    Therefore, when assuming $(T,\mathcal{V})$ to be strongly-linked, the unique path $t_iTt_i'$ in $T$ having $t_i$ and $t_i'$ as endpoints traverses an edge $f\in E(T)$ whose adhesion set $F_{f}$ has size at most $|\mathcal{P}| = |F_n| = 4n+1$. However, since $f\in E(t_iTt_i')$, the subgraphs $K_i$ and $K_i'$ are now separated also by $F_f$. Hence, due to each $P\in \mathcal{P}$ intersecting both $\delta(K_i)$ and $\delta(K_i')$ by construction, every edge from $F_f$ must then belong to precisely one path of this edge-disjoint family $\mathcal{P}$. In other words, $|F_f| = |\mathcal{P}| = 4n+1$.

    Finally, note that there are still unique rays $Q_n \in \mathcal{R}_n\setminus \{R_e: e\in F_n\}$ and $Q_n'\in \mathcal{R}_n'\setminus \{R_e': e\in F_n\}$, which are those from $\mathcal{R}_n$ and $\mathcal{R}_n'$ that start at the edges $u_0u_1$ and $v_0v_1$ respectively. In particular, as $E(P_e)\subseteq E(R_e)$ and $E(P_e')\subseteq E(R_e')$ for every $e\in F_n$, neither $Q_n$ nor $Q_n'$ intersect $F_f\subseteq \bigcup_{e\in F_n}E(P_e)\cup E(P_e')$. Therefore, since $Q_n\in \omega_n$ and $Q_n'\in \omega$ by the choices of $\mathcal{R}_n$ and $\mathcal{R}_n'$, we have $u_0\in C_E(F_f,\omega_n)$ and $v_0\in C_E(F_f,\omega)$. Similarly, $K_i\subseteq C_E(F_f, \omega_n)$ and $K_i'\subseteq C_E(F_f, \omega)$ because $\omega_n$ and $\omega$ live in $K_i$ and $K_i'$ respectively. Nevertheless, $K_i$ and $K_i'$ are subgraphs of $G$ contained in different connected components of $G\setminus F_f$, as argued in the previous paragraph. In other words, one concludes that $F_f$ separates $u_0$ and $v_0$, since $C_E(F_f,\omega_n)\neq C_E(F_f,\omega)$.
\end{proof}

Finishing this section, the above Lemma \ref{nthSeparator} reaches a direct proof of Theorem \ref{CounterexampleStronglyLinkedness}: if $(T,\mathcal{V})$ is a tree-cut decomposition of finite adhesion displaying $\omega$, then it can display $\omega_n$ only for finitely many values of $n\in\mathbb{N}$. After all, if $t_u,t_v\in T$ are the nodes of $T$ such that $u_0\in V_{t_u}$ and $v_0\in V_{t_v}$, then the adhesion sets of $(T,\mathcal{V})$ separating $u_0$ from $v_0$ are precisely those corresponding to the edges of the unique (finite) path $t_uTt_v$ in $T$ having $t_u$ and $t_v$ as endpoints.

\vspace{2.5cm}

\noindent
\textsc{Universität Hamburg, Department of Mathematics, Bundesstrasse 55 (Geomatikum), 20146 Hamburg, Germany}

\medskip

\noindent
\emph{Email address:}
\texttt{\{max.pitz,lucas.real\}@uni-hamburg.de}

\end{document}